\documentclass{amsart}

\usepackage{tikz-cd}
\usepackage{ amssymb }
\usepackage{hyperref}
\hypersetup{
    colorlinks=true,
    filecolor=magenta,      
    urlcolor=cyan,
}
\newtheorem{theorem}{Theorem}[section]
\newtheorem{lemma}[theorem]{Lemma}
\newtheorem{proposition}[theorem]{Proposition}

\theoremstyle{definition}
\newtheorem{definition}[theorem]{Definition}
\newtheorem{cor}[theorem]{Corollary}

\theoremstyle{remark}
\newtheorem{remark}[theorem]{Remark}

\numberwithin{equation}{section}

 \usepackage[OT2,T1]{fontenc}
  \DeclareSymbolFont{cyrletters}{OT2}{wncyr}{m}{n}
  \DeclareMathSymbol{\sha}{\mathalpha}{cyrletters}{"58}

\newcommand{\spec}{\text{Spec} \ }
\newcommand{\inj}{\hookrightarrow}

\newcommand{\R}{\mathbb{R}}
\newcommand{\A}{\mathbb{A}}

\newcommand{\n}{\mathfrak{n}}
\newcommand{\m}{\mathfrak{m}}
\newcommand{\p}{\mathfrak{P}}

\newcommand{\Q}{\mathbb{Q}}
\newcommand{\F}{\mathbb{F}}
\newcommand{\Z}{\mathbb{Z}}
\newcommand{\C}{\mathbb{C}}
\newcommand{\st}{\text{st}}
\newcommand{\gr}{\mathrm{gr}}

\newcommand{\ad}{\mathrm{ad}}

\newcommand{\gal}{\mathrm{Gal}}
\newcommand{\vl}{\mathrm{val}}

\newcommand{\Cor}{\text{Cor}}
\newcommand{\av}{\text{avoid}}
\newcommand{\suff}{\text{suff}}
\newcommand{\sym}{\text{Symm}}
\newcommand{\res}{\mathrm{Res}}

\begin{document}

\title{Potential Automorphy for $GL_n$}

\author{Lie Qian}
\address{Department of Mathematics, Stanford University}
\curraddr{Department of Mathematics, Stanford University}
\email{lqian@stanford.edu}
\thanks{}


%    General info

\date{April 19, 2021}

\dedicatory{}

\keywords{}

\begin{abstract}
We prove potential automorphy results for a single Galois representation $G_F\rightarrow GL_n(\overline{\Q}_l)$ where $F$ is a CM number field. The strategy is to use the $p,q$ switch trick and modify the Dwork motives employed in \cite{HSBT} to break self-duality of the motives, but not the Hodge-Tate weights. Another key result to prove is the ordinarity of certain $p$-adic representations, which follows from log geometry techniques. One input is the automorphy lifting theorem in \cite{tap}.
\end{abstract}

\maketitle

\section{Introduction}

In this paper we prove potential automorphy theorems for $n$-dimensional $l$-adic and residual representations of the absolute Galois group of an imaginary CM field. 

The precise statement of the theorem for residual representations is as following.

\begin{theorem}
\label{1.1}

Suppose $F$  is a CM number field, $F^{\mathrm{a}\mathrm{v}}$  is a finite extension of $F$ and $n\geq 2$  is a positive integer. Let $l$  be a prime number and suppose that
$$
\overline{r}: \mathrm{Gal}(\overline{F}/F)\rightarrow GL_n(\mathbb{F}_{l^{s}})
$$
 is a continuous semisimple representation. Then there exists a finite CM Galois extension $F'/F$  linearly disjoint from $F^{\mathrm{a}\mathrm{v}}$  over $F$ such that $\overline{r}\mid_{\mathrm{Gal}(\overline{F}/F')}$  is ordinarily automorphic.

\end{theorem}

We first breifly recall the definitions of the terms appearing in the theorem.

Recall that for $E$ any CM (or totally real) field, we could attach to any regular algebraic cuspidal automorphic representation $\pi$ of $GL_{n}(\mathbb{A}_{E})$ an $l$-adic Galois representation of $G_{E}$ satisfying certain local-global compatibility condition by the main theorem of \cite{HLLT}. 

More precisely, fix an isomorphism $\overline{\mathbb{Q}}_{l}\rightarrow \mathbb{C}$. For such a $\pi$, there is a unique continuous semisimple representation
$$
r_{l,\iota}(\pi):G_{E}\rightarrow GL_{n}(\overline{\mathbb{Q}}_{l})
$$
 such that, if $p\neq l$  is a rational prime above which $\pi$  and $E$  are unramified and if $v|p$  is a prime of $E$,  then $r_{l,\iota}(\pi)$  is unramified at $v$  
$$
r_{l,\iota}(\pi)|_{W_{E_{v}}}^{ss}=\iota^{-1}rec_{E_{v}}(\pi_{v}|\mathrm{det}|_{v}^{(1-n)/2})
$$
here $rec_{E_{v}}$  denotes the local Langlands correspondance for $E_{v}$ and $|^{ss}$ denotes the semisimplification.

\begin{definition}
For a $p$-adic local field $L$ and a continuous representation $\rho: G_L\rightarrow GL_n(\overline{\Q}_p)$, we say it is ordinary with regular Hodge-Tate weight if there exists a weight $\lambda=(\lambda_{\tau, i})\in (\{(a_1, \ldots, a_n)|a_1\geq\cdots\geq a_n\})^{\mathrm{H}\mathrm{o}\mathrm{m}(L,\overline{\mathbb{Q}}_{p})}=:(\mathbb{Z}_{+}^{n})^{\mathrm{H}\mathrm{o}\mathrm{m}(L,\overline{\mathbb{Q}}_{p})}$  such that there is an isomorphism:

\[
\rho\sim\left(
\begin{matrix}
\psi_{1}&\ast&\ast&\ast\\
0&\psi_{2}&\ast&\ast\\
\vdots&\ddots&\ddots&\ast\\
0&\cdots&0&\psi_{n}
\end{matrix}
\right)
\]
where for each $i=1,\ldots,n$  the character $\psi_{i}$ : $G_{L}\rightarrow\overline{\mathbb{Q}}_{p}^{\times}$  agrees with the character

$$
\sigma\in I_{L}\mapsto\prod_{\tau\in \mathrm{H}\mathrm{o}\mathrm{m}(L,\overline{\mathbb{Q}}_{p})}\tau(\mathrm{Art}_{L}^{-1}(\sigma))^{-(\lambda_{\tau,n-i+1}+i-1)}
$$
 on an open subgroup of the inertia group $I_{L}$.
\end{definition}

\begin{definition}
For a Galois representation $r: G_{E}\rightarrow GL_{n}(\overline{\mathbb{Q}}_{l})$ , we say it is automorphic if there exists a regular algebraic cuspidal automorphic representation $\pi$ such that $r\cong r_{l,\iota}(\pi)$ . And for a residual representation $\overline{r}$ : $G_{E}\rightarrow GL_{n}(\overline{\mathbb{F}}_{l})$ , we say it is automorphic if there exists a lift $r$ of $\overline{r}$ that is automorphic. We say it is ordinarily automorphic if there exists an automorphic lift $r$ which is potentially semistable and ordinary with regular Hodge-Tate weights as $G_{E_v}$ representation for all $v\mid l$.

\end{definition}

We also remark here that restricting to some $G_{F'}$ for an Galois extension $F'/F$ that avoids a prescribed finite extension $F^{\text{av}}$ of $F$ can ensure that the image $\overline{r}(G_{F'})$ does not shrink.

Combine our main theorem for residual representation Theorem \ref{1.1} with the automorphy lifting theorem 6.1.2 from \cite{tap} and the main result of \cite{qian}, we obtain a potential automorphy theorem for a single $l$-adic Galois representation into $GL_n$.

\begin{theorem}
\label{1.4}
 Suppose $F$  is a CM number field, $F^{\mathrm{a}\mathrm{v}}$  is a finite extension of $F$ and $n\geq 2$  is a positive integer. Let $l$  be a prime number. Fix an isomorphism $\iota: \overline{\mathbb{Q}}_{l}\rightarrow \mathbb{C}$ and suppose that
 \[
 r: G_F\rightarrow GL_n(\overline{\Q}_l)
 \]
 is a continuous representation satisfying the following condition:
 \begin{itemize}
     \item $r$ is unramified almost everywhere.
     \item For each place $v|l$ of $F$,  the representation $r|_{G_{F_{v}}}$  is potentially semistable, ordinary with regular Hodge-Tate weights $\lambda\in(\mathbb{Z}_{+}^{n})^{\mathrm{H}\mathrm{o}\mathrm{m}(F,\overline{\mathbb{Q}}_{p})}$.
     \item $\overline{r}$  is absolutely irreducible and decomposed generic (See \cite{tap} Definition 4.3.1). The image of $\overline{r}|_{G_{F(\zeta_{l})}}$  is enormous (See \cite{tap} Definition 6.2.28).
     \item There exists $\sigma\in G_{F}-G_{F(\zeta_{l})}$  such that  $\overline{r}(\sigma)$ is a scalar.
     \end{itemize}
     
Then there exists a finite CM Galois extension $F'/F$  linearly disjoint from $F^{\mathrm{a}\mathrm{v}}$  over $F$ such that $r\mid_{G_{F'}}$  is ordinarily automorphic.
\end{theorem}

Previously there are potential automorphy results for $\overline{r}$ and $r$ that take value in $GSp_{2n}$ (\cite{HSBT}) or more generally, the subgroup of $GL_{n}$ that preserves a nondegenerate form up to a scalar (\cite{BLGHT}). The strategy of proving the main theorems in this paper is based on the strategy of proving theorems in these paper. But there are also many crucial differences.

The main idea of the proof of Theorem \ref{1.1} is as the following. The prime $l$ is given. But we will choose some positive integer $N$ and another prime $l'$ with good properties. Note that this choice make certain arguments for the $l'$-related objects easier than their $l$-related counterpart.

Consider the Dwork family $Y\subset \mathbb{P}^{N-1}\times(\mathbb{P}^{1}\backslash (\mu_{N}\cup \{\infty\}))$ defined by the following equation:
$$
X_{1}^{N}+X_{2}^{N}+\ +X_{N}^{N}=NtX_{1}X_{2}\text{ . . . }X_{N}
$$
for the good $N$ we will choose.  The variety comes equipped with an action of the group
$$
H_{0}:=\{(\xi_{1},\ .\ .\ .\ \xi_{N})\in\mu_{N}^{N}:\xi_{1}\cdots\xi_{N}=1\}/\mu_{N}
$$
(see Section \ref{3}). After picking a character $\chi$ of $H_{0}$, we may consider the motives such that their $l$(or $l'$)-adic realisation is the $\chi$-eigenspace of the $(N-2)$-th(middle degree) etale cohomology of any fibre of this family with coefficients $\overline{\mathbb{Q}}_{l}$ (or $\overline{\mathbb{Q}}_{l'}$). We will denote such $l$(or $l'$)-adic cohomology of the fibre over the point $t$ as $V_{\lambda, t}$(or $V_{\lambda', t}$), where $\lambda, \lambda'$ is a place of $\Q(\zeta_N)$ above $l, l'$. Note that here we will choose a $\chi$ that is of a shape artificially made to break the self-duality of the motive. However, the self-duality shape of the Hodge-Tate weights will be preserved. In fact, they are a string of consecutive integers. We will try to find a point $t$ on the base defined over an extension field $F'$, such that the mod-$l$ residual Galois representation $V[\lambda]_t$ given by the fibre of the motive over $t$ is isomorphic to the $\overline{r}$ in the theorem, while the mod-$l'$ residual Galois representation $V[\lambda']_t$ given by the fibre of the motive over $t$ is isomorphic to $\overline{r}_{l',\iota}(\pi)$ for some known ordinarily automorphic representation $r_{l',\iota}(\pi)$, both as representation of $G_{F'}$ . If such a point exists, then we could apply ordinary automorphy lifting theorem 6.1.2 of \cite{tap} to see $V_{\lambda', t}$ is automorphic, and conclude that $V_{\lambda, t}$ is automorphic. Hence $\overline{r}$ is automorphic.

\ 

The above is a very rough summary of what we did in this paper. Let us be more precise now.

The first problem that is worth more explanantion is the existence of such a point $t$. Assume $\overline{r}$ and $\overline{r}_{l',\iota}(\pi)$ can be defined as representation over $k(\lambda)$ and $k(\lambda')$, where these are the residue fields for the places $\lambda$ and $\lambda'$ of $\Z[\zeta_N]$. The existence of such a point $t$ is guaranteed by a careful study of the moduli scheme that detects the isomorphisms between $\overline{r}\times \overline{r}_{l',\iota}(\pi)$ and the varying $V[\lambda]_t\times V[\lambda']_t$ as representation over $k(\lambda)\times k(\lambda')$, such that the top wedge of the isomorphism is fixed to be an a priori choice. Now the main property we use to prove the existence of such a point $t$ is the geometric connectivity of the moduli variety. And the geomoretic connectivity is in turn deduced from the result that the geometric monodromy map of this family surjects onto $SL_{n}(k(\lambda))\times SL_{n}(k(\lambda'))$, over which the fiber over $t$ of the moduli scheme is a torsor. The proof of this surjectivity result involves combinatorial arguments that precisely make use of the shape of the charater $\chi$ we choose. In contrast, we know that if we had chosen the $\chi$ to be of some nice self-dual form, then the image of the geometric monodromy map would be contained in some symplectic group $Sp_{n}(k(\lambda))\times Sp_{n}(k(\lambda'))$. We remark that in previous work, other authors have considered the moduli variety parametrizing similar isomorphisms but with the condition that certain alternating forms on both representation spaces need to be preserved, where the alternating form on the varying cohomology is induced by Poincare duality and the self-dual shape of the $\chi$ they chose.

In the above procedure, after showing that the geometric monodromy has image in $SL_{n}(k(\lambda))\times SL_{n}(k(\lambda'))$ (using  the shape of $\chi$), we see that the spaces $\wedge^n(V[\lambda]_t\times V[\lambda']_t)$ as characters of $G_F$ does not depend on the base point $t$. Thus to construct the moduli scheme, it suffices to construct a fixed isomorphism between $\wedge^n(\overline{r}\times \overline{r}_{l',\iota}(\pi))$ and $\wedge^n(V[\lambda]_t\times V[\lambda']_t)$ for any chosen $t$ an $F$ point of the base. However,   this isomorphism  does not a priori exist.  We get around this problem by restricting to a smaller $G_{F'}$ and twisting $\overline{r}\times \overline{r}_{l',\iota}(\pi)$ by a character $\overline{\chi}_1\times \overline{\chi}_2: G_{F'}\rightarrow k(\lambda)^\times\times k(\lambda')^\times$, so that we would like to construct the moduli scheme as the one detecting isomorphisms between $\left(\overline{r}\times \overline{r}_{l',\iota}(\pi)\right)\otimes\left(\overline{\chi}_1\times \overline{\chi}_2\right)$ and $V[\lambda]_t\times V[\lambda']_t$. Note that we need $\overline{\chi}_1\times \overline{\chi}_2$ to take value in exactly $k(\lambda)^\times\times k(\lambda')^\times$ because we want the fiber of the moduli scheme to be a torsor under the image $SL_{n}(k(\lambda))\times SL_{n}(k(\lambda'))$ of the geometric monodromy map because this is crucial to show the geometric connectivity of the moduli scheme.

Now choosing $t=0$, to construct an isomorphism between $\det\left(\left(\overline{r}\times \overline{r}_{l',\iota}(\pi)\right)\otimes\left(\overline{\chi}_1\times \overline{\chi}_2\right)\right)$ and $\det\left(V[\lambda]_t\times V[\lambda']_t\right)$,   amounts to taking an "$n$-th root" of the character $(\det V[\lambda]_0\times V[\lambda']_0)^{-1}\otimes\det(\overline{r}\times \overline{r}_{l',\iota}(\pi))$ as character valued in $k(\lambda)^\times\times k(\lambda')^\times$, where $V[\lambda]_0, V[\lambda']_0$ denotes the mod $l, l'$ cohomology of the fibre over $0$.    The first step to make this adjustment work is that we need 
\begin{itemize}
    \item $(\det V[\lambda]_0\times V[\lambda']_0)^{-1}\otimes\det(\overline{r}\times \overline{r}_{l',\iota}(\pi))$  has image in $(k(\lambda)^\times)^n\times (k(\lambda')^\times)^n$ as a $G_{F'}$ representation.
\end{itemize}
We remark that this condition is proved by a computation for the fibre over $0$, where there is a good description. The computation is done in \ref{npower}. Then, we will use Lemma \ref{2.1} to deduce that this condition above enables us to construct such an "$n$-th root" of character while also making sure that $F'$ satisfies certain linearly disjoint properties.

\ 

The second problem is that to apply ordinary automorphy lifting theorems, we also need to show $V_{\lambda, t}$ and $V_{\lambda', t}$ are both ordinary. 

The proof of $V_{\lambda', t}$ being ordinary is relatively easy. We just pick $t\in \A^1(F')$ that is $l'$-adicly close to $0$. Applying \ref{2.3}, we may check ordinarity via an examination of $D_{\mathrm{cris}}(V_{\lambda', t})$ . The comparison theorem identifies $D_{\mathrm{cris}}(V_{\lambda', t})$ and $D_{\mathrm{cris}}(V_{\lambda', 0})$, and hence  reduces the proof to the case $t=0$. In that case, $V_{\lambda', 0}$ actually splits into characters as a $G_{F'}$ representation.

To prove that $V_{\lambda, t}$ is ordinary is harder and relies heavily on the machinery of log geometry. This is the result of \cite{qian}. The idea is to choose $t\in \A^1(F')$ that has $l$-adic valuation $<0$. The intuition is that via the construction of a semistable model, the comparison theorems and an application of the log cristalline cohomology theory of Hyodo-Kato, we should have that  the operator $N$ acting on $D_{\mathrm{st}}(V_{\lambda, t})$ is identified with  the residue of the Gauss-Manin connection at $\infty$ of $V_B$ in the notation of Section \ref{3}, upon identifying the underlying space they act on. The latter is some sort of log of the monodromy around the point $\infty$ of $V_B\otimes\C$ and hence $N$ is maximally nilpotent because the monodromy is maximally unipotent by \ref{maxun}. Then ordinarity follows from a $p$-adic Hodge theoretic lemma.

The difference between the case of $l$ and $l'$ arises partially from the fact that $l$ is given but we may choose $l'$ arbitrarily as well.

\

Lastly the $\pi$ we use such that $r_{l',\iota}(\pi)$ is ordinarily automorphic is such that $r_{l',\iota}(\pi)$ is a symmetric tensor power of the Tate module of an elliptic curve over $\mathbb{Q}$.

\

With the above input and an awkward choice of the character $\chi$ of $H_{0}$, plus several technical algebraic number theory lemmas listed in Section 2, we will finally prove the theorem in Section 4.

\subsection*{Acknowledgement}
I would like to first thank Richard Taylor for encouraging me to think about the subject of this paper. I also want to thank him for all the helpful comments on the draft of this paper. I benefit a lot from the many interesting conversation with Richard Taylor, Brian Conrad, Weibo Fu, Ravi Vakil, Jun Su and Jack Sempliner during the preparation of this text.

\section{Several Lemmas}\label{2}

Let us first state the properties we will use throughout the paper regarding the notion of linearly disjoint fields. 
\begin{itemize}
    \item  If $A$ and $B$ are extensions of $C$ then $A$, $B$ linearly disjoint over $C$ implies $A
\cap B =C$, and the converse is true if $A$ or $B$ is finite Galois over $C$.

\item  If $A \supset B \supset C$ and $D\supset C$ with $A$ and $D$ linearly disjoint over $C$, then $A$
and $BD$ are linearly disjoint over $B$. In particular $A \cap BD = B$.

\end{itemize}

\begin{lemma}
\label{2.1}
 {\it For a CM} fi{\it eld} $F, a$ fi{\it nite CM Galois extension} $M/F, a$ fi{\it nite Galois extension} $F_{0}/\Q$ , a finite field $\mathbb{F}_{l^{r}}$ {\it containing all $n$-th roots of unity, and a character} $\chi$ : $G_{M}\rightarrow(\mathbb{F}_{l^{r}}^{\times})^{n}$, {\it there exists a} finite totally real Galois extension $L/\Q$ linearly disjoint with $F_0$ over $\Q$ and such that if we denote $F_{1}=LM$  , there exists a character $\psi$ : $G_{F_{1}}\rightarrow \mathbb{F}_{l^{r}}^{\times}$ {\it such that} $\psi^{n}=\chi|_{G_{F_{1}}}$. 

\end{lemma}

\begin{proof}
 Consider the long exact sequence associated to the following short exact sequence of $G_{M}$-module with trivial action:
 \[
\begin{tikzcd}
0\arrow[r]& \mathbb{Z}/m\mathbb{Z}\arrow[r]& \mathbb{F}_{l^{r}}^{\times}\arrow[r, "(\cdot)^n"]&(\mathbb{F}_{l^{r}}^{\times})^{m}\arrow[r]& 0
\end{tikzcd}
\]
where we write $n=l^{a}m, l\nmid m$, we have:
\[
\begin{tikzcd}
H^{1}(G_{M},\mathbb{F}_{l^{r}}^{\times})\arrow[r, "(\cdot)^n"]& H^{1}(G_{M},(\mathbb{F}_{l^{r}}^{\times})^{m})\arrow[r, "\delta"]&
H^{2}(G_{M},\mathbb{Z}/m\mathbb{Z})
\end{tikzcd}
\]

Now $\chi\in H^{1}(G_{M},\ (\mathbb{F}_{l^{r}}^{\times})^{m})$ . If we let $\tilde{\chi}=\delta(\chi)$ , we are reduced to finding a Galois CM extension $F_{1}\supset M$ of the form $F_1=LM$ for some $L$ linearly disjoint with $F_{0}$ over $\Q$ such that the obstruction $\tilde{\chi}$ is killed by the restriction map $H^{2}(G_{M},\mathbb{Z}/m\mathbb{Z})\rightarrow H^{2}(G_{F_{1}}, \mathbb{Z}/m\mathbb{Z})$.

Consider the map $H^{2}(G_{M}, \mathbb{Z}/m\mathbb{Z})\rightarrow\prod_{v}H^{2}(G_{M_{v}},\ \mathbb{Z}/m\mathbb{Z})$ given by restriction. The image actually lands in $\bigoplus_{v}H^{2}(G_{M_{v}}, \mathbb{Z}/m\mathbb{Z})$ because any element in $H^{2}(G_{M}, \mathbb{Z}/m\mathbb{Z})$ is inflated from some $\phi\in H^{2}(\gal(M'/M), \mathbb{Z}/m\mathbb{Z})$ for some $M'/M$ a finite extension and for those primes $v$ of $M$ that is unramified in $M'$, the image of $\phi$ in $H^{2} (G_{M_{v}}, \mathbb{Z}/m\mathbb{Z})$ by restriction actually lands in $H^{2}(\gal(M_{v}^{\mathrm{n}\mathrm{r}}/M_{v}), \mathbb{Z}/m\mathbb{Z})$ , which is $0$ since the cohomological dimension of $\hat{\mathbb{Z}}$ is 1.

\

The first step is to take an CM extension $F_{2}/M$ that is of the form $L_2M$ for a totally real $L_2$ Galois over $\Q$ that is linearly disjoint with $F_{0}$ over $\Q$, such that in the following commutative diagram, the image of $\tilde{\chi}$ in the upper right corner is $0$:
\[
\begin{tikzcd}
H^{2}(G_{F_{2}}, \mathbb{Z}/m\mathbb{Z})\arrow[r]&
\bigoplus_{w}H^{2}(G_{F_{2,w}}, \mathbb{Z}/m\mathbb{Z})\\
H^{2}(G_{M}, \mathbb{Z}/m\mathbb{Z})\arrow[r]\arrow[u]&
\bigoplus_{v}H^{2}(G_{M_{v}}, \mathbb{Z}/m\mathbb{Z})\arrow[u]
\end{tikzcd}
\]

Let $\bigoplus\tilde{\chi}_{v}$ be the image of $\tilde{\chi}$ in $\bigoplus_{v}H^{2}(G_{M_{v}}, \mathbb{Z}/m\mathbb{Z})$ . If we can take a CM extension $F_{2}/M$ of the above form such that for any $v$ with $\tilde{\chi}_{v}\neq 0$ and $w|v$ a place of $F_{2}$, $\zeta_{m}\in F_{2,w}$ and $m|[F_{2,w}: M_{v}(\zeta_{m})]$, then the image of $\tilde{\chi}_{v}$ restricting to $H^{2}(G_{F_{2,w}}, \mathbb{Z}/m\mathbb{Z}) $ is 0, since $H^{2}(G_{M_{v}(\zeta_{m})},\mathbb{Z}/m\mathbb{Z})\cong H^{2}(G_{M_{v}(\zeta_{m})},\mu_{m})\cong\frac{1}{m}\mathbb{Z}/\mathbb{Z}$, and the restriction map $\frac{1}{m}\mathbb{Z}/\mathbb{Z}\cong H^{2}(G_{M_{v}(\zeta_{m})}, \mu_{m})\rightarrow H^{2}(G_{F_{2,w}},\mu_{m})\cong\frac{1}{m}\mathbb{Z}/\mathbb{Z}$ is multiplication by $[F_{2,w}:M_{v}(\zeta_{m})]$. 

We can construct such an extension $F_{2}/M$ coming from $L_2/\Q$ linearly disjoint with $F_{0}$ over $\Q$  with prescribed local behavior for a finite number of primes $v$ of $M$  as the following:

Let $S_{1}$ be the set of rational primes lying under the primes $v$ of $M$ such that $\widetilde{\chi}_{v}\neq 0$. Let $S_{2}=$\{$\infty$\} and $S=S_{1}\cup S_{2}$. For each $q\in S_{1}$, let $M_{q}$ denote the composite of the image of all embeddings $\tau$ : $M\inj\overline{\Q}_{q}$. We fix an extension $E_{q}/M_{q}(\zeta_{m})$ of order divisible by $m$ and Galois over $\Q_q$. Now we apply Lemma 4.1.2 of \cite{CHT} to $F_0/\Q$ and the set of primes $S$ with prescribed local behavior:

\begin{itemize}

\item $E_{q}/\Q_{q}$ for all $q\in S_{1}$

\item Trivial extension $\R/\R$ for $\infty\in S_{2}$
 
 \end{itemize}

We get a finite totally real Galois extension $L_2/\Q$ that is linearly disjoint with $F_0$ over $\Q$ and that for any $w$ a place of $L_2$ over a $q\in S_1$, $(L_2)_{w}\cong E_q$ we defined above over $\Q_q$. Take $F_{2}=ML_2$. For any prime $v$ of the field $M$ with $\tilde{\chi}_{v}\neq 0$ and $v'$ a prime of $F_{2}$ over $v$, then $q=v'_{|\Q}\in S_{1}$, thus $F_{2,v'}\supset (L_2)_{v'\mid_{L_2}}\supset E_{q}$ and so $\gal(F_{2,v'}/M_{v}(\zeta_{m}))$ is of order divisible by $m$. So we have constructed the desired $L_2$ and $F_2$.

\ 

Now for any number field $F$ and $G_{F}$ module $A$, let $\sha^{i}(F, A)$ be
$$
\mathrm{k}\mathrm{e}\mathrm{r}(H^{i}(G_{F}, A)\rightarrow\prod_{v}H^{i}(G_{F_{v}}, A))
$$
where the product is over all places $v$ of $F$ (so is every product that follows).

Thus the first step yields a finite CM Galois extension $F_{2}/F$ containing $M$ that comes from some $L_2/\Q$ as described above such that the image $\tilde{\chi}_{1}$ of  $\tilde{\chi}$ in $H^{2}(G_{F_{2}}, \mathbb{Z}/m\mathbb{Z})$ actually lies in $\sha^{2}(F_{2}, \mathbb{Z}/m\mathbb{Z})$ .

\

The second step is to analyze $\sha^{2}(F_{2}, \mathbb{Z}/m\mathbb{Z})$ and kill it after some further CM extension $F_{1}/F_{2}$ where $F_1=L_1L_2M$, with $L_1$ totally real Galois over $\Q$ that would be specified later and such that $L:=L_1L_2$ is linearly disjoint with $F_0$ over $\Q$.

Poitou-Tate duality(cf. \cite{Neu} Theorem 8.6.7) gives a perfect pairing
$$
\langle\cdot,\ \cdot\rangle:\sha^{2}(F_{2}, \mathbb{Z}/m\mathbb{Z})\times \sha^{1}(F_{2}, \mu_{m})\rightarrow \mathbb{Q}/\mathbb{Z}
$$
satisfying the following compatibility for any finite extension $F_{1}/F_{2}$ and $x\in \sha^{1}(F_{1}, \mu_{m})$, $y\in \sha^{2}(F_{2}, \mathbb{Z}/m\mathbb{Z})$ :
$$
\langle x, {\rm Res}(y)\rangle=\langle \mathrm{C}\mathrm{o}\mathrm{r}(x), y\rangle
$$
So we now choose such an extension $F_{1}/F_{2}$ such that $\mathrm{C}\mathrm{o}\mathrm{r}(x)=0$, $\forall x\in \sha^{1}(F_{1}, \mu_{m})$, then by the perfectness, ${\rm Res}(y)=0$, $\forall y\in \sha^{2}(F_{2}, \mathbb{Z}/m\mathbb{Z})$.

Write $m=2^{r} \prod_{i=1}^{s}p_{i}^{r_{i}}$. Decompose $\sha^{1}(F_{2}, \mu_{m})=\sha^{1}(F_{2}, \mu_{2^{r}})\times\prod_{i=1}^{s}\sha^{1}(F_{2}, \mu_{p_{i}^{r_{i}}})$. The following lemma is basically Theorem 9.1.9 of \cite{Neuk}.

\begin{lemma}
\label{2.2}
 {\it For any number} field $F$, $\sha^{1}(F, \mu_{p^{r}})=0$  or  $\mathbb{Z}/2\mathbb{Z}$. {\it The later case could happen only when} $p=2$.

{\it In any case}, $\sha^{1}(F, \mu_{p^{r}}) =\sha^{1}(F(\mu_{p^{r}})/F, \mu_{p^{r}})$ ({\it de}fi{\it ned in the proof}) .

\end{lemma}

\begin{proof}
 Set $K=F(\mu_{p^{r}})$ . We have the following commutative diagram where each row and column (except the left column) are exact:
\[
\begin{tikzcd}
&0\arrow[r]
&H^{1}(G_{K}, \mu_{p^{r}})\arrow[r]
&\prod_{w}H^{1}(G_{K_{\omega}}, \mu_{p^{r}})\\
0\arrow[r] 
&\sha^{1}(F, \mu_{p^{r}})\arrow[r]\arrow[u]
&H^{1}(G_{F}, \mu_{p^{r}})\arrow[r]\arrow[u]
&\prod_{v}H^{1}(G_{F_{v}}, \mu_{p^{r}})\arrow[u]\\
0\arrow[r]
&\sha^{1}(K/F, \mu_{p^{r}})\arrow[r]\arrow[u, hook]
&H^{1}(\gal(K/F), \mu_{p^{r}})\arrow[r]\arrow[u, hook]
&\prod_{v}H^{1}(\gal(K_{w}/F_{v}), \mu_{p^{r}})\arrow[u, hook]
\end{tikzcd}
\]

where $\sha^{1} (K/F, \mu_{p^{r}})$ is defined by the exactness of the bottom row and the top row is exact because an element in the kernel corresponds to an cyclic extension of $K$ of order dividing $p^r$ that splits at all primes $w$ of $K$, which has to be trivial. (Again $\mu_{p^r}=\Z/p^r\Z$ as a $G_{K}$ module and $H^1$ is just $\mathrm{Hom}$.)

A diagram chasing gives that $\sha^{1} (F, \mu_{p^{r}}) =\sha^{1}(K/F, \mu_{p^{r}})$. By Proposition 9.1.6 of \cite{Neuk}, $H^{1} (\mathrm{G}\mathrm{a}\mathrm{l}(K/F), \mu_{p^{r}}) =0$ except when
\begin{itemize}
    \item  $p=2$, $r\geq 2$
    \item  and $-1$ is in the image of $\mathrm{G}\mathrm{a}\mathrm{l}(K/F)\rightarrow(\mathbb{Z}/2^{r}\mathbb{Z})^{\times}$
\end{itemize}

In this case, $H^{1} (\mathrm{G}\mathrm{a}\mathrm{l}(K/F), \mu_{2^{r}})=\mathbb{Z}/2\mathbb{Z}$. As a subspace of $H^{1}(\mathrm{G}\mathrm{a}\mathrm{l}(K/F), \mu_{2^{r}})$, $\sha^{1} (F, \mu_{p^{r}}) =\sha^{1}(K/F, \mu_{p^{r}})=0$ or $\mathbb{Z}/2\mathbb{Z}$.

\end{proof}

Recall the relation $\mathrm{Cor}\circ{\rm Res}=[F_{1}:F_{2}]$ and the commutative diagram:
\[
\begin{tikzcd}
H^{1}(G_{F_{2}}, \mu_{2^{r}})\arrow[r, "\mathrm{Res}"]
&H^{1}(G_{F_{1}}, \mu_{2^{r}})\\
H^{1} (\mathrm{G}\mathrm{a}\mathrm{l}(F_{2}(\mu_{2^{r}})/F_{2}), \mu_{2^{r}}) \arrow[r, "\mathrm{Res}"]\arrow[u, hook]
&H^{1}(\mathrm{G}\mathrm{a}\mathrm{l}(F_{1}(\mu_{2^{r}})/F_{1}), \mu_{2^{r}})\arrow[u, hook]
\end{tikzcd}
\]

The bottom row is an isomorphism if we pick $F_{1}$ linearly disjoint with $F_{2}(\mu_{2^{r}})$ over $F_{2}$. If this is the case and $2\mid [F_{1}:F_{2}]$, then by Lemma \ref{2.2}
\begin{equation}
\begin{split}
\mathrm{Cor}(\sha^{1}(F_{1}, \mu_{m}))&=\mathrm{C}\mathrm{o}\mathrm{r}(\sha^{1}(F_{1}, \mu_{2^{r}}))\\
&=\mathrm{C}\mathrm{o}\mathrm{r}(\sha^{1}(F_{1}(\mu_{2^{r}})/F_1, \mu_{2^{r}}))\\
&\subset \mathrm{C}\mathrm{o}\mathrm{r}(H^{1}(\mathrm{G}\mathrm{a}\mathrm{l}(F_{1}(\mu_{2^{r}})/F_{1}), \mu_{2^{r}}))\\
&=\mathrm{C}\mathrm{o}\mathrm{r}({\rm Res}(H^{1}(\mathrm{G}\mathrm{a}\mathrm{l}(F_{2}(\mu_{2^{r}})/F_{2}), \mu_{2^{r}})))\\
&=[F_{1}:F_{2}]\cdot H^{1}(\mathrm{G}\mathrm{a}\mathrm{l}(F_{2}(\mu_{2^{r}})/F_{2}), \mu_{2^{r}})\\
&=0
\label{cor}
\end{split}
\end{equation}

Here when we apply $\Cor$ to some group, we always mean $\Cor$ applied to the image of this group in $H^1(G_{F_1}, \mu_{2^r})$.

Now we construct an $L_1$ such that the associated extension $F_{1}/F_{2}$ satisfies the property that $F_{1}$ is linearly disjoint with $F_{2}(\mu_{2^{r}})$ over $F_{2}$ and $2\mid [F_{1}:F_{2}]$. Choose a rational prime $p$ and a local field $E_{p}$ Galois over $\Q_p$ that contains $M'_p$, the composite of the image of all embeddings $\tau: F_2\inj \overline{\Q}_p$, and is of order divisible by $2$ over it. We again apply Lemma 4.1.2 of \cite{CHT} to the extension $F_0F_2(\mu_{2^r})/\Q$ and the set of primes $S$ consisting of $p$ and $\infty$ with prescribed local behavior:   
\begin{itemize}
    \item $E_{p}/\Q_p$
    \item Trivial extension $\R/\R$ for the place $\infty$
\end{itemize}

We get a totally real Galois extension $L_1/\Q$ linearly disjoint with $F_0F_2(\mu_{2^r})$ over $\Q$. The associated $F_1=F_2L_1=ML_2L_1$. Then $2\mid [F_1:F_2]$ because for any place $v$ of $F_1$ above $p$, $(F_1)_v\supset (L_1)_{v\mid_{L_1}}\cong E_p\supset (F_2)_{v\mid_{F_2}}$ and the last inclusion is of order divisible by $2$. The property that $F_1=F_2L_1$ and $F_2(\mu_{2^r})$ are linearly disjoint over $F_2$ follows from the fact that $L_1$ and $F_2(\mu_{2^r})$ are linearly disjoint over $\Q$.  Now $L_1$ and $F_0F_2$ are linearly disjoint  over $\Q$ implies that $L_1L_2$ and $F_0F_2$ are linearly disjoint over $L_2$. Hence  $L_{1}L_2\cap F_{0}=L_1L_2\cap F_{0}F_{2}\cap F_{0}=L_{2}\cap F_{0}=\Q$.

We conclude that the image of $\tilde{\chi}$ in $H^{2}(G_{F_{1}}, \mathbb{Z}/m\mathbb{Z})$ is $0$ by (\ref{cor}) and thus we can take an $n$-th root of $\chi|_{G_{F_{1}}}$ for $F_{1}=LM\supset M$ and $L=L_1L_2$ we constructed above finite totally real Galois over $\Q$ and linearly disjoint with $F_{0}$ over $\Q$.

\end{proof}

\begin{lemma}
\label{2N}
 {\it Let $l$ be a rational prime. Given any positive integer} $s$ {\it and a finite set of rational primes} $S$, {\it we can} fi{\it nd a positive integer} $N$ {\it not divisible by any primes in} $S$ {\it and} $l$, {\it and satisfying}:

\begin{itemize}
    \item  {\it Let} $r$ {\it be the smallest positive integer such that} $N\mid l^{r}-1$, {\it then} $s\mid r.$

    \item {\it When} $r$ {\it is even}, $N\nmid l^{r/2}+1$

 \end{itemize}
\end{lemma}

\begin{proof}
 Factorize $s$ as $s=2^{a_{0}}\displaystyle \prod_{i=1}^{m}p_{i}^{a_{i}}$. View $p_{0}=2$. We will construct a sequence of pairwise coprime integers $M_{i}$ (not divisible by any rational prime in $S$) and a sequence of integers $t_{i}$ with $t_{i}\geq a_{i}$ for $i=0,1,\ldots, m$, such that $M_{i}\mid l^{r}-1$ if and only if $p_{i}^{t_{i}}\mid r$. Set $N=\displaystyle \prod_{i=0}^{m}M_{i}$ and consider the order of $l$ in $(\mathbb{Z}/N\mathbb{Z})^{\times}\cong \displaystyle \prod_{i=1}^{m}(\mathbb{Z}/M_{i}\mathbb{Z})^{\times}$, we see that the first condition is satisfied. For the second condition (if $a_{0}>0$), we need to make $M_{0}$ satisfy the following extra property:
$$
M_{0}\nmid l^{2^{t_{0}-1}}+1
$$
Now we work on $i=0$ first. Take $t_{0}>a_{0}$ large enough such that $t_{0}>2$ and for each rational prime $q\in S\cup\{2\}$, one of the following holds:

(1) $q \nmid l^{2^{k}}-1$ for any $k>0$

(2) $q\mid l^{2^{t_{0}-3}}-1$

The fact $t_{0}>2$ gives $l^{2^{t_{0}-2}}+1\equiv 2$ mod $4$ and $l^{2^{t_{0}-1}}+1\equiv 2$ mod $4$ (or are both odd when $l=2$). Thus we may choose  an odd prime divisor $ A$ of $l^{2^{t_{0}-2}}+1$ and  an odd prime divisor $B$ of $l^{2^{t_{0}-1}}+1$. We deduce that
$$
AB\mid l^{2^{t_{0}}}-1,\ B\nmid l^{2^{t_{0}-1}}-1,\ A\nmid l^{2^{t_{0-}1}}+1
$$
Take $M_{0}=AB$. Thus the smallest $r$ such that $M_{0}\mid l^{r}-1$ is $2^{t_0}$ and $M_{0}\nmid l^{2^{t_{0}-1}}+1$. Also, for $q\in S$, if $q\mid M_{0}$, then $q\mid l^{2^{t_{0}-2}}+1$ or $q\mid l^{2^{t_{0}-1}}+1$. In either case (1) won't happen, so $q\mid l^{2^{t_{0}-3}}-1$. Thus $q=2$. And this gives a contradiction with $AB$ being odd. We have constructed an $M_{0}$ with the property stated above.

Now we inductively construct $M_{i}$ and $t_{i}\geq a_{i}$ such that
\begin{itemize}
    \item  $M_{i}$ is not divisible by any rational primes in $S_{i}=\{p_{i}\}\cup S \cup \{$rational prime divisors of $M_{j}$ for $j<i\}\cup\{l\}\cup \{2\}$.
    \item The order of $l$ in $(\mathbb{Z}/M_{i}\mathbb{Z})^{\times}$ is $p_{i}^{t_{i}}$.
\end{itemize}

Choose $t_{i}>a_{i}$ large enough, such that $l^{p_i^{t_i-2}}>p_{i}$ and for each rational prime $q\in S_{i}$, one of the following holds:
\begin{enumerate}
    \item  $q\nmid l^{p_{i}^{k}}-1$ for any $k>0$
    \item $q\mid l^{p_{i}^{t_{i}-2}}-1$
\end{enumerate}

If $q\in S_{i}$ and $q\mid l^{p_{i}^{t_{i}-1}(p_i-1)}+\ldots +l^{p_{i}^{t_{i}-1}}+1$, then $q\mid l^{p_{i}^{t_{i}}}-1$ and so (1) cannot hold. Hence $q\mid l^{p_{i}^{t_{i}-2}}-1$. Thus $l^{p_{i}^{t_{i}-1}(p_{i}-1)}+\ldots +l^{p_{i}^{t_{i}-1}}+1\equiv p_{i}$ mod $q$. We see $q=p_{i}$. In this case, $l^{p_{i}^{t_{i}-1}}\equiv (l^{p_{i}^{t_{i}-2}})^{p_i}\equiv (1+p_iu)^{p_i}\equiv 1$ mod $p_{i}^{2}$ for some integer $u$ , so $l^{p_i^{t_{i}-1}(p_{i}-1)}+\ldots +l^{p_{i}^{t_{i}-1}}+1\equiv p_{i}$ mod $p_{i}^{2}$ and $p_{i}^{2}\nmid l^{p_{i}^{t_{i}-1}(p_{i}-1)}+\ldots +l^{p_{i}^{t_{i}-1}}+1$. Thus, the only prime in $S_{i}$ that divides $l^{p_{i}^{t_{i}-1}(p_{i}-1)}+\ldots +l^{p_{i}^{t_{i}-1}}+1$ is $p_{i}$ and only to the first order. We may take an odd prime divisor $M_{i}$ of $l^{p_{i}^{t_{i}-1}(p_{i}-1)}+\ldots +l^{p_{i}^{t_{i}-1}}+1$($>p_i$) with $M_{i}\not\in S_{i}$. Now $M_{i}\mid l^{p_{i}^{t_{i}}}-1$, but if $M_{i}\mid l^{p_i^{t_i-1}}-1$, then $l^{p_{i}^{t_{i}-1}(p_{i}-1)}+\ldots +l^{p_{i}^{t_{i}-1}}+1\equiv p_{i}$ mod $M_{i}$. So $M_{i}=p_{i}$ giving a contradiction. The two condition on $M_{i}$ is thus satisfied.

Now take $N=\prod_{i=0}^{m}M_{i}$ as promised. The smallest positive integer $r$ such that $N\mid l^{r}-1$ is $2^{t_0}\prod_{i=1}^{m}p_{i}^{t_{i}}$, a multiple of $s$. For the second condition, if $m>0$, i.e. there are odd prime divisors of $s$, then $M_{1}\mid l^{2^{t-1}\prod_{i=1}^{m}p_{i}^{t_{i}}}-1$ yields a contradiction with $N\mid l^{r/2}+1$. If $m=0$, then the construction stops at the first step and we have seen that $M_{0}=N\nmid l^{2^{t_{0}-1}}+1$.

\end{proof}

The following lemmas are taken from \cite{BLGHT} Lemma 2.2 with a little modification for the second one. These lemma will be used to prove that certain reprsentations coming from the Dwork motive are ordinary.

\begin{lemma}
\label{2.3}
  {\it Suppose that a} $\in(\mathbb{Z}^{n})^{\mathrm{H}\mathrm{o}\mathrm{m}(F,\overline{\mathbb{Q}}_{l}),+}$ {\it and that}
$$
r: \gal(\overline{F}/F)\rightarrow GL_{n}(\overline{\mathbb{Q}}_{l})
$$
{\it is crystalline at all primes} $v\mid l$. {\it We think of} $v$ {\it as a valuation} $v$ : $F_{v}^{\times}\twoheadrightarrow \mathbb{Z}$. {\it If} $\tau$: $F\rightarrow\overline{\mathbb{Q}_{l}}$ {\it lies above v, suppose that}
$$
\dim_{\overline{\mathbb{Q}}_{l}}\gr^{i}(r\otimes_{\tau,F_{v}}B_{\mathrm{D}\mathrm{R}})^{\gal(\overline{F}_{v}/F_{v})}=0
$$
{\it unless} $i=a_{\tau,j}+n-j$ {\it for some} $j=1,\ldots,n$, {\it in which case}
$$
\dim_{\overline{\mathbb{Q}}_{l}}\gr^{i}(r\otimes_{\tau,F_{v}}B_{\mathrm{D}\mathrm{R}})^{\gal(\overline{F}_{v}/F_{v})}=1
$$
{\it For} $v\mid l$, {\it let} $\alpha_{v,1},\ldots, \alpha_{v,n}$ {\it denote the roots of the characteristic polynomial of} $\phi^{[F_{v}^{0}:\mathbb{Q}_{l}]}$ {\it on}
$$
(r\otimes_{\tau,F_{v}^{0}}B_{\mathrm{c}\mathrm{r}\mathrm{i}\mathrm{s}})^{\gal(\overline{F}_{v}/F_{v})}
$$
{\it for any} $\tau: F_{v}^{0}\inj\overline{\mathbb{Q}}_{l}$. (Here $F_v^0$ is the maximal unramified subextension in $F_v$. {\it This characteristic polynomial is independent of the choice of} $\tau$.) {\it Let} $\mathrm{v}\mathrm{a}\mathrm{l}_{v}$ {\it denote the valuation on} $\overline{\mathbb{Q}}_{l}$ {\it normalized by} $\vl_{v}(l)=v(l)$. ({\it Thus} $\mathrm{v}\mathrm{a}\mathrm{l}_{v}\circ\tau=v$ {\it for any} $\tau:F_{v}\inj\overline{\mathbb{Q}}_{l}$.) {\it Arrange the} $\alpha_{v,i}$'{\it s such that}
$$
\vl_{v}(\alpha_{v,1})\geq \vl_{v}(\alpha_{v,2})\geq\ldots\geq \vl_{v}(\alpha_{v,n})
$$
{\it Then} $r$ {\it is ordinary of weight} $a$ {\it if and only if for all} $v\mid l$ {\it and all} $i=1,\ldots, n$ {\it we have}
$$
\mathrm{v}\mathrm{a}\mathrm{l}_{v}(\alpha_{v,i})=\sum_{\tau}(a_{\tau,i}+n-i)
$$
{\it where} $\tau$ {\it runs over embedding} $F\inj\overline{\mathbb{Q}}_{l}$ {\it above v}.
\end{lemma}

\begin{remark}
We will use $D_{\mathrm{cris},\tau}(r)$, $D_{\st,\tau}(r)$ to denote $(r\otimes_{\tau,F_{v}^{0}}B_{\mathrm{c}\mathrm{r}\mathrm{i}\mathrm{s}})^{\gal(\overline{F}_{v}/F_{v})}$, $(r\otimes_{\tau, F_v^0}B_{\st})^{\mathrm{Gal}(\overline{F}_v/F_v)}$ resp. for any $p$-adic representation $r$ and embedding $\tau$ as above.
\end{remark}

\section{Dwork Motives}
\label{3}

In this section, $l$ can be any prime, $n$ be any integer  $\geq 2$ and $N$ is an integer that is

\begin{itemize}
    \item odd, not divisible by any prime factors of $ln$.
    \item $N>100n+100$
\end{itemize}
but note that the case $n>2$ and $n=2$ differs a little bit, in that there will be a slight change of the category where the objects we considered lie in.

We assume in this section that $F$ is a CM number field containing $\zeta_{N}$.

We will modify the construction and argument in section 4 of \cite{BLGHT} to fit the situation where no self-duality holds.

Let $T_{0}=\mathbb{P}^{1}-(\{\infty\}\cup\mu_{N})/\mathbb{Z}[1/N]$ with coordinate $t$ and $Y\subset \mathbb{P}^{N-1}\times T_{0}$ be a projective family defined by the following equation:
$$
X_{1}^{N}+X_{2}^{N}+\ +X_{N}^{N}=NtX_{1}X_{2}\text{ . . . }X_{N}
$$

$\pi:Y\rightarrow T_{0}$ is a smooth of relative dimension $N-2$. We will write $Y_{s}$ for the fiber of this family at a point $s$. Let $H=\mu_{N}^{N}/\mu_{N}$ where the second $\mu_N$ embeds diagonally and
$$
H_{0}:=\{(\xi_{1},\ldots, \xi_{N})\in\mu_{N}^{N}:\xi_{1}\cdots\xi_{N}=1\}/\mu_{N}\subset H
$$
Over $\mathbb{Z}[1/N, \zeta_{N}]$ there is an $H$ action on $Y$ by:
$$
(\xi_{1},\ldots, \xi_{N})(X_{1},\ldots, X_{N}, t)=(\xi_{1}X_{1},\ldots, \xi_{N}X_{N}, (\xi_{1}\cdots\xi_{N})^{-1}t)
$$
Thus $H_{0}$ acts on every fibre $Y_{s}$, and $H$ acts on $Y_{0}$.

Fix $\chi$ a character $H_{0}\rightarrow\mu_{N}$ of the form:
$$
\chi\ ((\xi_{1},\ldots, \xi_{N}))=\prod_{i=1}^{N}\xi_{i}^{a_{i}}
$$
where
\[
(a_{1},\ldots, a_{N})=(1,2,4,5,\ldots, (N-n+2)/2, (N+n-2)/2,\ldots, N-5, N-4, N-3,0,0,\ldots, 0)
\]
when $n>2$ is odd,
\[
(a_{1},\ldots, a_{N})=(1,2,3,4,5,7,8,\ldots, (N-n+3)/2, (N+n-3)/2, (N+n-1)/2,\ldots, N-4,0,0,\ldots, 0)
\]
when $n>2$ is even, and
\[
(a_{1},\ldots, a_{N})=(1,2,\ldots (N-3)/2, 0, 0, 0, (N+3)/2, \ldots, N-1)
\]
when $n=2$.

Note that $3, N-2, N-1$ do not occur in $(a_1,\ldots, a_N)$ when $n>2$ is odd, $6, N-3, N-2, N-1$ do not occur in $(a_1,\ldots, a_N)$ when $n>2$ is even,  and there are $n+1$ of $0$s in $(a_1,\ldots, a_N)$ in these cases ($n>2$). 

This character is well-defined because $\sum_{i=1}^{N}a_{i}\equiv 0$ mod $N$.

Let $(b_1, \ldots, b_n)$ be mutually distinct residue classes in $\Z/N\Z$ such that $b_i+a_j\neq 0$ in $\Z/N\Z$ for any $j\in\{1,\ldots, N\}$. Hence we have the following expression:
\[
\{b_{1},\ldots, b_{n}\}=\left\{
\begin{array}{ll}
1, 2, (N-n+4)/2, (N-n+6)/2,\ldots, (N+n-6)/2, (N+n-4)/2, N-3& n>2 \  \mathrm{odd}\\
1, 2, 3, (N-n+5)/2, (N-n+7)/2,\ldots, (N+n-7)/2, (N+n-5)/2, N-6& n>2\  \mathrm{even}\\
(N-1)/2, (N+1)/2& n=2
\end{array}\right.
\]
when $n=3,4$,  this is interpretted as $\{b_{1},\ldots, b_{n}\}=\{1,2, N-3\}, \{1,2,3, N-6\}$
  respectively.)

We have the following combinatorial property for the set $\{b_1, \ldots, b_n\}\subset \Z/N\Z$.

\begin{lemma}
\label{3.5}
Let $n>2$. {\it Consider} $\{b_{1},\ldots, b_{n}\}$ {\it as a subset of} $\{0,1,\ldots,N-1\}\cong \mathbb{Z}/N\mathbb{Z}$. {\it If for some} $\alpha\in(\mathbb{Z}/N\mathbb{Z})^{\times}$, $\{\alpha b_{1},\ldots,\alpha b_{n}\}=\{b_{1},\ldots,b_{n}\}$ {\it holds, then} $\alpha=1.$

\end{lemma}

\begin{proof}
If $n$ is even, then $1\in\{b_{1},\ldots, b_{n}\}$ and so $\alpha\in\{b_{1},\ldots, b_{n}\}$. If $\alpha=2$, then $2\in\{b_{1},\ldots, b_{n}\}$ and so $4\in\{b_{1},\ldots, b_{n}\}$. But by the assumption $N>100n+100$, $3<4<(N-n+5)/2$, so that $4\notin\{b_{1},\ldots, b_{n}\}$. Thus $\alpha\neq 2$. Same argument $(3<9,36<(N-n+5)/2)$ shows that $\alpha\neq 3, N-6$. If $n=4$, we are done. For $n\geq 6$, if $\alpha=(N+x)/2$ for some $x$ odd and $x\in[-n+5, n-5]$, then $2\alpha=x\in\{b_{1},\ldots, b_{n}\}$, but by assumption on $N$ and $n$, we have $(N+n-5)/2<N-n+5\leq N-1$ and $1\leq n-5< (N-n+5)/2$. Therefore $\{1, 3,\ldots, n-5, N-n+5, N-n+7,\ldots, N-1\}\cap\{b_{1},\ldots, b_{n}\}\subset \{1, 3\}$ and so $x=1$ or $3$. In either case, if $\alpha=(N+x)/2\in\{b_{1},\ldots, b_{n}\}$, then $(N-x)/2\in\{b_{1},\ldots, b_{n}\}$, so $-x^{2}/4\equiv(N+x)/2\cdot(N-x)/2\in\{b_{1},\ldots, b_{n}\}$, thus $N-1\in\{4b_{1},\ldots, 4b_{n}\}$ or $N-9\in\{4b_{1},\ldots, 4b_{n}\}$. Viewed as a subset of the representatives $\{0,1,\ldots, N-1\}$, $\{4b_{1},\ldots, 4b_{n}\}\subset\{2, 6 ,\ldots,2n-10, N-(2n-10), N-(2n-14), \ldots, N-2\}\cup \{4,8,12, N-24\}$. But $N-1, N-9>2n-10,  12$, so they must lies in $\{N-(2n-10), N-(2n-14),\ldots, N-2\}\cup\{N-24\}$, a contradiction. 

If $n$ is odd, again $1\in\{b_{1},\ldots, b_{n}\}$ and so $\alpha\in\{b_{1},\ldots, b_{n}\}$. If $\alpha=2$, then again $4\in\{b_{1},\ldots, b_{n}\}$. But $2<4<(N-n+4)/2$ gives a contradiction. Similarly $\alpha\neq N-3$ because $2<9<(N-n+4)/2$. If $n=3$ we are done. For $n\geq 5$, we have $\alpha=(N+x)/2$ for some $x$ odd and $x\in [-n+4, n-4]$. Then $2\alpha=x\in\{b_{1},\ldots, b_{n}\}$, but by assumption on $N$, we have $(N+n-4)/2<N-n+4\leq N-1$ and $1\leq n-4< (N-n+4)/2$. Therefore $\{1, 3,\ldots, n-4, N-n+4, N-n+6,\ldots, N-1\}\cap\{b_{1},\ldots, b_{n}\}\subset\{1, N-3\}$ and so $x=1$ or $-3$. In either case, if $\alpha=(N+x)/2\in\{b_{1},\ldots, b_{n}\}$, then $(N-x)/2\in\{b_{1},\ldots, b_{n}\}$, so $-x^{2}/4\equiv(N+x)/2\cdot(N-x)/2\in\{b_{1},\ldots, b_{n}\}$, thus $N-1\in\{4b_{1},\ldots, 4b_{n}\}$ or $N-9\in\{4b_{1},\ldots, 4b_{n}\}$. Viewed as a subset of the representatives $\{0,1,\ldots, N-1\}$, $\{4b_{1},\ldots, 4b_{n}\}\subset\{2, 6, \ldots,2n-8, N-(2n-8), N-(2n-12), \ldots, N-2\}\cup\{4,8,N-12\}$. But $N-1, N-9>2n-8, 8$, so they must lie in $\{N-(2n-8), N-(2n-12), \ldots, N-2\}\cup\{N-12\}$, a contradiction.
\end{proof}

One reason for this choice is that when $n>2$ we want to avoid making the set $\{a_{1},\ldots, a_{N}\}$ self-dual so that Lemma \ref{3.2} and Lemma \ref{3.4}  hold. And the above lemma will be crucial to avoid self duality. On the contrary, if the set $\{a_{1},\ldots, a_{N}\}$ is chosen to be self-dual , the $V[\lambda]$ defined below would be a Galois representation that takes values in $GSp_n$.

For any prime $\lambda$ of $\mathbb{Z}[1/2N, \zeta_{N}]$ of residue characteristic $l$, we define the lisse sheaf $V_{\lambda}/(T_{0}\times{\rm Spec} \mathbb{Z}[1/2Nl, \zeta_{N}])_{et}$ by:
$$
V_{\lambda}=(R^{N-2}\pi_{\ast}\mathbb{Z}[\zeta_{N}]_{\lambda})^{\chi,H_{0}}
$$
when $n>2$. When $n=2$, we use the same formula to define the object $U_\lambda$ for a prime $\lambda$ of $\mathbb{Z}[1/2N, \zeta_{N}]^+$, following the notation of \cite{BLGHT}.

Similarly, for any nonzero ideal $\n$ of $\mathbb{Z}[1/2N, \zeta_{N}]$ of norm $M$, we can define the lisse sheaf $V[\n]/(T_{0}\times{\rm Spec} \mathbb{Z}[1/2NM, \zeta_{N}])_{et}$ by:
$$
V[\n]=(R^{N-2}\pi_{\ast}(\mathbb{Z}[\zeta_{N}]/\n))^{\chi,H_{0}}
$$
when $n>2$, and we use the same formula to define the object $U[\n]$ for a nonzero ideal $\n$ of $\mathbb{Z}[1/2N, \zeta_{N}]^+$.

Since $H$ acts on $Y_{0}$, we have the following decomposition:
$$
V_{\lambda,0}=\bigoplus_{i=1}^{N}V_{\lambda,i},  \  V[\n]_{0}=\bigoplus_{i=1}^{N}V_{i}[\n]
$$
here $V_{\lambda,i}$ and $V_{i}[\n]$ are the subspace of $V_{\lambda,0}, V[\n]$ where $H$ acts by the character $\chi_{i}$:
$$
\xi\rightarrow\prod_{j=1}^{N}\xi_{j}^{a_{j}+i}
$$

Again, we write same decomposition in the case $n=2$ as into $U_{\lambda, i}$ and $U_i[\n]$.

Fix an embedding $\tau:\mathbb{Q}(\zeta_{N})\inj\mathbb{C}$  such that $\tau(\zeta_{N})=e^{2\pi i/N}$. Let $\tilde{\pi}: Y(\mathbb{C})\rightarrow T_{0}(\mathbb{C})$ denote the base change of $\pi$ along $\tau$, viewed as a map of complex analytic spaces and $V_{B}$ be the locally constant sheaf over $T_{0}(\mathbb{C})$:
$$
V_{B}=(R^{N-2}\tilde{\pi}_{\ast}\mathbb{Z}[\zeta_{N}])^{\chi,H_{0}}
$$
when $n>2$. And we denote the same object in the case $n=2$ as $U_B$.

Let $\tau$ also denote the induced base change $(T_{0})_{\mathbb{C}}\rightarrow T_{0}\times \spec \mathbb{Z}[1/2NMl,\zeta_N]$. Under previous notation, $V_{B}\otimes_{\mathbb{Z}[\zeta_{N}]}\mathbb{Z}[\zeta_{N}]_{\lambda}$ corresponds to $\tau^{*}V_{\lambda}$ under the equivalence between locally constant analytic $\Z[\zeta_N]_\lambda$-sheaf on $T_0(\C)$ and locally constant etale  $\Z[\zeta_N]_\lambda$-sheaf on $(T_{0})_{\mathbb{C}}$. 

Similarly, $
V_{B}\otimes_{\mathbb{Z}[\zeta_{N}]}\mathbb{Z}[\zeta_{N}]/\n$ corresponds to $\tau^{*}V[\n]
$ under the equivalence between locally constant analytic $\Z[\zeta_N]/\n$-sheaf on $T_0(\C)$ and locally constant etale  $\Z[\zeta_N]/\n$-sheaf on $(T_{0})_{\mathbb{C}}$.

Similar relation holds when $n=2$, see \cite{BLGHT} section 4.

\

Let $T_{0}'=\mathbb{P}^{1}-\{0,1, \infty\}$ with coordinate $t'$ and $Y'\subset \mathbb{P}^{N-1}\times T_{0}'$ be a projective family defined by the following equation:
$$
X_{1}^{\prime N}+X_{2}^{\prime N}+\cdots+t^{\prime-1}X_{N}^{\prime N}=NX_{1}'X_{2}'\cdots X_{N}'
$$
Then $t'\mapsto t^{N}$ gives an $N$-fold Galois covering $T_{0}-\{0\}\rightarrow T_{0}'$  and  $X_{1}'\rightarrow X_{1},  X_{2}'\rightarrow X_{2}\ldots, X_{N}'\rightarrow tX_{N}$ identifies the pullback of $\pi': Y '\rightarrow T_{0}'$ along this covering with $\pi: Y-Y_{0}\rightarrow T_{0}-\{0\}$. %In other words, over $\spec \mathbb{Q}(\zeta_N)$ , if we let $\mu_{N}$ act on the family $Y-Y_{0}\rightarrow T_{0}-\{0\}$ by
%$$
%\xi(X_{1}, X_{2},\ldots, X_{N}, t)=(X_{1}, X_{2},\ldots, \xi X_{N}, \xi^{-1}t)
%$$
%then $Y'=(Y-Y_{0})/\mu_{N}$ and $T_{0}'=(T_{0}-\{0\})/\mu_{N}$. 
Over $\mathbb{Z}[1/N, \zeta_{N}]$, $H_{0}$ acts on $Y'$ by
$$
(\xi_{1},\ldots,\xi_{N})(X_{1}',\ldots, X_{N}', t')=(\xi_{1}X_{1}',\ldots, \xi_{N}X_{N}', t')
$$
This $H_{0}$ action is compatible with the $H_{0}$ action on $Y-Y_{0}$.

\

Let $\widetilde{\pi}': Y'(\mathbb{C})\rightarrow T_{0}'(\mathbb{C})$ be the base change of $\pi'$ along $\tau$ viewed as a map of complex analytic spaces and let $V_{B}'=(R^{N-2}\widetilde{\pi}_{\ast}'\mathbb{Z}[\zeta_{N}])^{\chi,H_{0}}$ be a locally constant sheaf over $T_{0}'(\mathbb{C})$. Then the pullback of $V_{B}'$ along the covering $T_{0}(\mathbb{C})-\{0\}\rightarrow T_{0}'(\mathbb{C})$ is naturally identified with $V_{B}$ over $T_{0}(\mathbb{C})-\{0\}$.

Fix a nonzero base point $t\in T_{0}(\mathbb{C})$ and let $t'$ be its image in $T_{0}'(\mathbb{C})$ . Now we study the image of the monodromy representation:
$$
\rho_{t}:\pi_{1}(T_{0}(\mathbb{C}), t)\rightarrow GL(V_{B,t})
$$

We in turn consider the monodromy representation:
$$
\rho_{t'}:\pi_{1}(T_{0}'(\mathbb{C}), t')\rightarrow GL(V_{B,t'}')
$$

\begin{proposition}
 {\it The sheaves}  $V_{\lambda}, V[\n], V_{B}, V_{B}'$ {\it are locally free over}  $\mathbb{Z}[\zeta_{N}]_{\lambda}, \mathbb{Z}[\zeta_{N}]/\n$, $\mathbb{Z}[\zeta_{N}], \mathbb{Z}[\zeta_{N}]$ {\it of rank} $n$ {\it respectively}.

\end{proposition}

\begin{proof}
Locally freeness follows because the family is smooth and proper. For the rank part, one only need to check at the fibre over $0$ and apply Proposition I.7.4 of \cite{DMOS}.

\end{proof}

Similar relation holds in the case $n=2$, see \cite{BLGHT}
section 4. The different aspect for the case $n=2$ is that we may use the locally constant sheaves $V_\lambda, V[\n], V_B$ defined there with coefficients $\Z[\zeta_N]^+_\lambda, \Z[\zeta_N]^+/\n, \Z[\zeta_N]^+$ respectively, such that  $V_\lambda\otimes_{\Z[\zeta_N]^+}\Z[\zeta_N]$, $V[\n]\otimes_{\Z[\zeta_N]^+}\Z[\zeta_N]$, $V_B\otimes_{\Z[\zeta_N]^+}\Z[\zeta_N]$ are isomorphic to $U_\lambda$, $U[\n]$, $U_B$ respectively. Now we consistently work with $V_\lambda, V[\n], V_B$, regardless of whether $n>2$ or $n=2$.

We already have the counterpart of Lemma \ref{3.4} as provided by Corollary 4.7 of \cite{BLGHT}. i.e. $\overline{\rho}_{t}(\pi_{1}(T_{0}(\mathbb{C}), t))=SL(V_{B,t}/\lambda)$. So we focus on the case $n>2$ until the end of the proof of Lemma \ref{3.4}.

Let $\gamma_{0}, \gamma_{1}, \gamma_{\infty}$ be the loop around $0,1, \infty$, generating $\pi_{1}(T_{0}'(\mathbb{C}), t')$ subject only to the relation $\gamma_{0}\gamma_{1}\gamma_{\infty}=1$. Here we let $\gamma_{0}$ be such oriented that its image in $\mathrm{G}\mathrm{a}\mathrm{l}(T_{0}(\mathbb{C})-\{0\}/T_{0}'(\mathbb{C}))$ is $e^{2\pi i/N}=\tau(\zeta_{N})$

\begin{lemma}
\label{3.2}

\begin{enumerate}
    \item $\rho_{t'}(\gamma_0)$ has characteristic polynomial $\prod_{j=1}^{n}(X-\zeta_{N}^{b_{\mathrm{j}}})$ {\it where}

\item $\rho_{t'}(\gamma_{\infty})$ {\it has characteristic polynomial} $(X-1)^{n}.$

\item $\rho_{t'}(\gamma_{1})$ {\it is a transvection, i.e}.: {\it it is unipotent and} $\mathrm{K}\mathrm{e}\mathrm{r}(\rho_{t'}(\gamma_{1})-1)$ {\it has dimension} $n-1$.
\end{enumerate}

\end{lemma}

\begin{proof}
(1) The action of $\gamma_{0}$ on $V_{B,t'}'$, is equivalent to the $\zeta_{N}$ action on $V_{B,0}$, which is the scalar multiplication by $\zeta_{N}^{i}$ on the $\chi_{i}$ eigenspace of $V_{B,0}$. By proposition I.7.4 of \cite{DMOS}, the $\chi_{i}$-eigenspaces are nonzero if and only if $0\notin\{i+a_1,\ldots,i+a_N\}$, i.e. $i\in\{b_{1},\ldots, b_{n}\}$, in which case the eigenspaces are all of rank 1. Hence the expression of the characteristic polynomial of $\rho_{t'}(\gamma_{0})$ follows.

\ 

(2) Suppose $Z_0$ is the variety $T(X_1^N+X_2^N+\cdots+X_N^N)=X_1X_2\cdots X_N$ contained in $\mathbb{P}^{N-1}\times \A^1$. We use $p$ to denote the projection $Z_0\rightarrow \A^1$. So it suffice to show the monodromy around $0$ of the larger local system $R^{N-2}p_\ast\C$ has charateristic polynomial a power of $(X-1)$.  We apply Lemma 2.1 of \cite{qian} base changed to $\C$ via $W(k)[T, U^\pm]\rightarrow \C[T]$, $U\mapsto 1, T\mapsto T$ and a fixed isomorphism of $\overline{W(k)[\frac{1}{p}]}\cong \C$, to conclude that there is a blowup $\mathbb{X}$ of $Z_0$ that is an isomorphism outside the fiber over $0$ and is semistable over the base $\A^1$. Note that we call a map semistable if the corresponding normal crossing divisor is reduced and does not have self crossing throughout this paper.

Thus, by the vanishing cycle technique used to prove local monodromy theorem (cf. \cite{ill} 2.1), we see that such monodromy is unipotent. Note that this makes use of the fact that our normal crossing divisor is reduced.

\ 

(3) The proof is the same as part 2 of Lemma 4.3 in \cite{BLGHT}. 
\end{proof}

Now we study the image of the monodromy map. Let $\lambda$ be a prime of $\mathbb{Z}[\zeta_{N}]$ (of $\mathbb{Z}[\zeta_{N}]^+$ when $n=2$) of characteristic $l$  and $\overline{\rho}_{t'}: \pi_{1}(T_{0}'(\mathbb{C}), t')\rightarrow GL(V_{B,t'}/\lambda)$, $\overline{\rho}_{t}:
\pi_{1}(T_{0}(\mathbb{C}), t)\rightarrow GL(V_{B,t}/\lambda)$ be the reduction of $\rho_{t'}$ , $\rho_{t}$ by $\lambda$ respectively.

We first give a description of $\rho_{t'}$ by Lemma \ref{3.2} and the following lemma.

\begin{lemma}
\label{3.3}
{\it Let} $\rho$ {\it be the representation} $\rho:\pi_{1}(T_{0}'(\mathbb{C}), t')\rightarrow GL_{n}(\mathbb{Z}[\zeta_{N}])$ {\it sending} $\gamma_{0}$ {\it to} $B^{-1}, \gamma_{\infty}$ {\it to} $A$, {\it and} $\gamma_{1}$ {\it to} $BA^{-1}$, {\it where}
$$
A=\ \left(\begin{array}{lllll}
0 & 0 & \cdots & 0 & -A_{n}\\
1 & 0 & \cdots & 0 & -A_{n-1}\\
0 & 1 & \cdots & 0 & -A_{n-2}\\
 &  & \ddots &  & \\
0 & 0 & \cdots & 1 & -A_{1}
\end{array}\right)
$$
$$
B=\ \left(\begin{array}{lllll}
0 & 0 & \cdots & 0 & -B_{n}\\
1 & 0 & \cdots & 0 & -B_{n-1}\\
0 & 1 & \cdots & 0 & -B_{n-2}\\
 &  & \ddots &  & \\
0 & 0 & \cdots & 1 & -B_{1}
\end{array}\right)
$$
{\it and} $A_{i}, B_{i}\in \mathbb{Z}[\zeta_{N}]$ {\it are the coe}ffi{\it cients of the expansions}:
$$
(X\ -1)^{n}=X^{n}+A_{1}X^{n-1}+\cdots+A_{n}
$$
$$\prod_{i=1}^{n}(X-\zeta_{N}^{-b_i})=X^{n}+B_{1}X^{n-1}+\cdots+B_{n}
$$
{\it Then as representation into} $GL_{n}(\mathbb{C})$, $\rho_{t'}$ {\it and} $\rho$ {\it are equivalent}.

\end{lemma}

\begin{proof}
See Theorem 3.5 of \cite{BH}. Observe also that $\rho(\gamma_{1})=BA^{-1}$ has the form

$$\left(\begin{array}{lllll}
C_{n} & 0 & \cdots & 0 & 0\\
C_{n-1} & 1 & \cdots & 0 & 0\\
 &  & \ddots &  & \\
C_{2} & 0 & \cdots & 1 & 0\\
C_{1} & 0 & \cdots & 0 & 1
\end{array}\right)
$$

with all the $C_{i}\in \mathbb{Z}[\zeta_{N}]$. 
\end{proof}

Note that the matrix $A$ actually has minimal polynomial $(X-1)^n$ and is conjugate to $\rho_{t'}(\gamma_{\infty})$. We have the following corollary.

\begin{cor}
\label{maxun}
$\rho_{t'}(\gamma_{\infty})$ has minimal polynomial $(X-1)^n$ and hence so is the image of the monodromy around $\infty$ under $\rho_t$.
\end{cor}

Let $\overline{\rho}:\pi_{1}(T_{0}'(\mathbb{C}),\ t')\rightarrow GL_{n}(\mathbb{Z}[\zeta_{N}]/\lambda)$ be the reduction of $\rho$ with respect to $\lambda$. (Following the argument of proposition 3.3 of \cite{BH}) Then if $\overline{\rho}$ has block upper-triangular form when base changed to $\overline{k(\lambda)}$, we see $\overline{\rho}(\gamma_{1})-1$ would vanish on one of the two blocks since it is a transvection, so that the eigenvalue of $\overline{\rho}(\gamma_{0})$ and $\overline{\rho}(\gamma_{\infty})$ would be the same on that block, which gives a contradiction because none of the $b_i$ is $0$. Thus $\overline{\rho}$ is absolutely irreducible. Let $\overline{\rho}_{t'}:\pi_{1}(T_{0}'(\mathbb{C}),\ t')\rightarrow GL(V_{B,t},/\lambda)$ be the reduction of $\rho_{t'}$ by $\lambda$. It has the same trace with $\overline{\rho}$ by Lemma \ref{3.3}. So their semisimplification are equivalent and thus they are equivalent and $\overline{\rho}_{t'}$ is absolutely irreducible.

\begin{lemma}
\label{3.4}
{\it Assume the residue} fi{\it eld} $k(\lambda)$ {\it of} $\lambda$ {\it is} $\mathbb{F}_{l^{r}}$(So $r$ is the smallest integer such that $N\mid l^r-1$). {\it Under the assumption that} $N\nmid l^{r/2}+1$ {\it if} $r$ {\it is even and $n>2$, we have that}  $\overline{\rho}_{t'}(\pi_{1}(T_{0}'(\mathbb{C}), t'))=SL(V_{B,t}',/\lambda)$ {\it and} $\overline{\rho}_{t}(\pi_{1}(T_{0}(\mathbb{C}), t))=SL(V_{B,t}/\lambda)$.

\end{lemma}

\begin{proof}
The case when $n=2$ is already resolved by Lemma 4.6 of \cite{BLGHT}. We now focus in the case $n>2$.

Let $H$ be the normal subgroup of $\pi_{1}(T_{0}'(\mathbb{C}), t')$ generated by $\gamma_{1}$. Then $\pi_{1}(T_{0}'(\mathbb{C}),\ t')/H$ is cyclic, and is generated by $\gamma_{0}H$ or $\gamma_{\infty}H$. Therefore the index $[\overline{\rho}_{t'}(\pi_{1}(T_{0}'(\mathbb{C}),\ t')) : \overline{\rho}_{t'}(H)]$ divides both the order of $\overline{\rho}_{t'}(\gamma_{0})$ and $\overline{\rho}_{t'}(\gamma_{\infty})$. The former is a divisor of $N$ and the latter is an $l$-power, thus  $\overline{\rho}_{t'}(\pi_{1}(T_{0}'(\mathbb{C}), t'))=\overline{\rho}_{t'}(H)$. So $\overline{\rho}_{t'}(\pi_{1}(T_{0}'(\mathbb{C}), t'))$ is generated by transvections, hence by the main theorem of \cite{SZ}, $\overline{\rho}_{t'}(\pi_{1}(T_{0}'(\mathbb{C}),\ t'))$ is conjugate in $GL_{n}(k(\lambda))$ to one of the groups $SL_{n}(k), Sp_{n}(k)$ or $SU(n, k)$ for some subfield $k\subset k(\lambda)$. Here $SU(n, k)$ is defined when $[k:\mathbb{F}_{l}]$ even:
$$
SU(n,\ k):=\{g\in SL_{n}(k): \sigma(g)^{t}g=1_{n}\}
$$
where $\sigma$ is the unique order $2$ element in $\mathrm{G}\mathrm{a}\mathrm{l}(k/\mathbb{F}_{l})$. We want to show $\overline{\rho}_{t'}(\pi_{1}(T_{0}'(\mathbb{C}), t'))= SL_{n}(k(\lambda))$ by excluding the other cases.

If $\overline{\rho}_{t'}(\pi_{1}(T_{0}'(\mathbb{C}), t'))$ is conjugate in $GL_{n}(k(\lambda))$ to one of the groups $SL_{n}(k)$, $Sp_{n}(k)$ or $SU(n, k)$ for some proper subfield $k\subsetneqq k(\lambda)$, then there exists a nontrivial $\sigma\in \mathrm{G}\mathrm{a}\mathrm{l}(k(\lambda)/\mathbb{F}_{l})$ that preserve the eigenvalues for any elements in $\overline{\rho}_{t'}(\pi_{1}(T_{0}'(\mathbb{C}), t'))$. Consider $\overline{\rho}_{t'}(\gamma_{0})$, this would contradict Lemma \ref{3.5}.

\ 

If $\overline{\rho}_{t'}(\pi_{1}(T_{0}'(\mathbb{C}), t'))$ is conjugate to $Sp_{n}(k(\lambda))$ , by Proposition 6.1 of \cite{BH}, we have
$$
\{b_{1},\ldots, b_{n}\}=\{-b_{1},\ldots, -b_{n}\}
$$
which contradicts Lemma \ref{3.5}.

\ 

If $\overline{\rho}_{t'}(\pi_{1}(T_{0}'(\mathbb{C}), t'))$ is conjugate to $SU(n, k(\lambda))$, then we are in the situation $[k(\lambda): \mathbb{F}_{l}]$ is even. Assume $k(\lambda)=\mathbb{F}_{l^{2s}}$. Take the eigenvalue of both sides of the equation $\sigma(\overline{\rho}_{t'}(\gamma_{0}))=({}^{t}\overline{\rho}_{t'}(\gamma_{0}))^{-1}$, we have
$$
\{l^{s}b_{1},\ldots, l^{s}b_{n}\}=\{-b_{1},\ldots, -b_{n}\}
$$
By Lemma \ref{3.5}, we must have $l^s\equiv -1$ mod $N$. This contradicts the condition.

Thus $\overline{\rho}_{t'}(\pi_{1}(T_{0}'(\mathbb{C}), t'))=SL_{n}(k(\lambda))$ . View $\overline{\rho}_t$ as defined on $\pi_{1}(T_{0}(\mathbb{C})- \{0\}, t)$ via the surjection $\pi_{1}(T_{0}(\mathbb{C})-\{0\}, t)\rightarrow\pi_{1}(T_{0}(\mathbb{C}), t)$. Since $\pi_{1}(T_{0}(\mathbb{C})-\{0\}, t)\vartriangleleft \pi_{1}(T_{0}'(\mathbb{C}), t')$ with quotient group cyclic of order $N$, we have $\overline{\rho}_{t}(\pi_{1}(T_{0}(\mathbb{C})-\{0\},\ t))\vartriangleleft SL_{n}(k(\lambda))$ with quotient cyclic of order dividing $N$. Now as the only cyclic composition factor of $SL_{n}(k(\lambda))$ have order dividing $n$, we see $\overline{\rho}_{t}(\pi_{1}(T_{0}(\mathbb{C}), t))=\overline{\rho}_{t}(\pi_{1}(T_{0}(\mathbb{C})-\{0\}, t))=SL_{n}(k(\lambda))$ .

\end{proof} 

Given any nonzero ideal $\n$ of  $\mathbb{Z}[1/2N, \zeta_{N}]$ (of $\mathbb{Z}[1/2N, \zeta_{N}]^+$ when $n=2$) and any finite free rank-$n$ $\mathbb{Z}[\zeta_{N}]/\n$($\mathbb{Z}[\zeta_{N}]^+/\n$ when $n=2$)-module $W$ with a continuous $G_{F}$-action, we can view $W$ as a lisse sheaf on $(\spec F)_{et}$. Now $\wedge^{n}V[\n]$ is a lisse sheaf over $(T_{0})_{F}$ of rank 1, and the associated monodromy representation $\det\rho:\pi_{1}(T_{0}, t)\rightarrow GL(\wedge^{n}V[\n]_{t})$ restricted to $\pi_{1}^{\mathrm{g}\mathrm{e}\mathrm{o}\mathrm{m}}(T_{0}, t)$ is trivial by Abhyankar's lemma and $\det(\gamma_{0})=\det(\gamma_{1})=\det(\gamma_{\infty})=1$. Thus $\det\rho$ factors through $\pi_{1}(\spec F)=G_{F}$.

Suppose we are given an isomorphism of lisse sheaf over $(T_{0})_{F}$ via some prescribed isomorphism of $G_F$ characters:
$$
\phi:\bigwedge^n W_{(T_{0})_{F}}\rightarrow\bigwedge^n V[\n]
$$
Let $\phi_{S}$ denote the base change of $\phi$ to some scheme $S$ over $(T_{0})_{F}$. Define the moduli functor $T_{W}$ as following:
\[
T_{W}: (\mathrm{S}\mathrm{c}\mathrm{h}/(T_{0})_{F}) \rightarrow (\mathrm{Sets})
\]
\[
S\mapsto\{\psi\in \mathrm{I}\mathrm{s}\mathrm{o}\mathrm{m}_{S}(W_{S}, V[n]_{S}): \wedge^{n}\psi=\phi_{S}\}
\]
It is reprsentable by a smooth $T_{W}/(T_{0})_{F}$.

\begin{proposition}
\label{3.6}
 {\it Under the notation and assumption above, if} $\n=\p_{1}\p_{2}$, {\it where} $\p_{1}, \p_{2}$ {\it are two prime ideals of} $\mathbb{Z}[\zeta_{N}]$ {\it having different residue characteristic} $l_{1}, l_{2}$ ({\it prime to} $N$) {\it respectively. If each of the} $l_{i}$ {\it satisfy  the following condition}:

\begin{itemize}

    \item   if $n>2$ the smallest positive $r$ {\it such that} $N\mid l_{i}^{r}-1$ {\it is even, then} $N\nmid l_{i}^{r/2}+1$. 
    
\end{itemize}
and $\mathrm{max}\{l_1, l_2\}>10$, {\it then} $T_{W}$ {\it is geometrically connected}.

\end{proposition}

\begin{proof}
Since $\pi_{1}(T_{0}(\mathbb{C}), t)\rightarrow SL(V_{B,t}/\p_1)$ and $\pi_{1}(T_{0}(\mathbb{C}), t)\rightarrow SL(V_{B,t}/\p_2)$ are surjective by Lemma \ref{3.4} and our condition, by Goursat Lemma we see that there exist isomorphic quotient $\phi: SL_{n}(\mathbb{F}_{l_{1}^{r}})/H_1\cong SL_{n}(\mathbb{F}_{l_{2}^{s}})/H_2$ such that the image of $\pi_{1}(T_{0}(\mathbb{C}), t)$ in $SL(V_{B,t}/\n)$ is the preimage of the diagonal $\{(t, \phi(t))\in SL_{n}(\mathbb{F}_{l_{1}^{r}})/H_1\times SL_{n}(\mathbb{F}_{l_{2}^{s}})/H_2\}$ under the natural quotient map. Here we let the residue field of $\p_{1}, \p_{2}$ be $\mathbb{F}_{l_{1}^{r}}, \mathbb{F}_{l_{2}^{s}}$ respectively.

Assume without loss of generality that $l_1>10$. Then the only proper normal subgroups of $SL_{n}(\mathbb{F}_{l_1^{r}})$ are contained in its center and the quotient group $PSL_{n}(\mathbb{F}_{l_1^{r}})$ is a simple group. Thus if $SL_{n}(\mathbb{F}_{l_{1}^{r}})/H_1$ is not trivial, then it must have a Jordan-Holder factor isomorphic to $PSL_{n}(\mathbb{F}_{l_1^{r}})$. Since $l_1>10$, any Jordan-Holder factor of $SL_{n}(\mathbb{F}_{l_{2}^{s}})$ with $l_2\neq l_1$ is not isomorphic to $PSL_{n}(\mathbb{F}_{l_1^{r}})$ by checking the duplication relation in the case $A_n$ of the classification of finite simple groups.(cf. \cite{spgp}). This contradiction gives us $SL_{n}(\mathbb{F}_{l_{1}^{r}})/H_1=1$ and the map $\pi_{1}(T_{0}(\mathbb{C}), t)\rightarrow SL(V_{B,t}/\n)$ is surjective.

Hence for any $t\in T_{0}(\mathbb{C})$ and any two geomereic points of $T_W$ above it which correspond to two isomorphisms $\psi_{1}, \psi_{2}: W\rightarrow V[\n]_{t}$ that respects $\phi$(not necessarily respecting any Galois action because the points are geometric, hence such points always exist), we can pick a path $\gamma\in\pi_{1}^{\mathrm{g}\mathrm{e}\mathrm{o}\mathrm{m}}(T_{0}, t)$ such that its image under the monodromy map is $\psi_{2}\circ\psi_{1}^{-1}$. Going along $\gamma$ induces a path in $T_{W}(\mathbb{C})$ connecting $\psi_{1}$ and $\psi_{2}$ (viewed as points in $T_{W}(\mathbb{C}))$, so geometrically connectedness follows.

\end{proof}

\begin{lemma}
\label{npower}
Let $\zeta_N\in F$, $\lambda$ is a prime of $\Q(\zeta_N)$ ($\Q(\zeta_N)^+$ when $n=2$) over the fixed prime $l$ and $k(\lambda)$ be the residue field,  then viewing $\det V[\lambda]$ as a $G_F$ representation as explained above, we have $\det V[\lambda](G_F)\subset (\mathbb{F}_{l^{2}}^{\times}k(\lambda)^{\times})^{n}$.
\end{lemma}

\begin{proof}
By the analysis before, it suffices to calculate the $G_{F}$ action on $V_{\lambda,0}=\oplus_{i=1}^{N}V_{\lambda,i}=\oplus_{i=1}^{N}H^{N-2}(Y_{0},\ \mathbb{Z}[\zeta_{N}]_{\lambda})^{\chi_{i},H}$, where $Y_{0}$ is the Fermat hypersurface $X_{1}^{N}+\cdots+X_{N}^{N}=0$ in $\mathbb{P}^{N-1}$ and $\chi_{i}: H\rightarrow\mu_{N}$ is a character defined by:
$$
\xi=(\xi_{1},\ldots,  \xi_{N})\mapsto\prod_{j=1}^{N}\xi_{j}^{a_{j}+i}
$$
By \cite{DMOS} Proposition 1.7.10,  $V_{\lambda,i}\neq 0$ (in fact $1$-dimensional) only when $i\in\{b_{1},\ldots, b_{n}\}$, and $\mathrm{F}\mathrm{r}\mathrm{o}\mathrm{b}_{v}$ acts on it by a scalar $q^{-1}\displaystyle \prod_{j=1}^{N}g(v, a_{j}+i)$ , where $v$ is any place of $F$ whose residue characteristic does not divide $N$ or $l$, $q=\# k(v)$, and $g(v, a)$ (for $a\in \mathbb{Z}/N\mathbb{Z}$) is the Gauss sum defined with respect to an fixed additive character $\psi$ : $\mathbb{F}_{q}\rightarrow(\overline{\mathbb{Q}}_{l})^{\times}$:
\[
g(v, a)=-\sum_{x\in \F_{q}^{\times}}t(x^{\frac{1-q}{N}})^{a}\psi(x)
\]
here we also fix an isomorphism $t$ from the group of $N$-th roots of unity in $\mathbb{F}_{q}^\times$ and the group of $N$-th roots of unity in $\overline{\mathbb{Q}}_{l}$. We remark that each $g(v, a)$ depends on the choice of $\psi$ but $q^{-1}\displaystyle \prod_{j=1}^{N}g(v, a_{j}+i)$ does not. 

Thus $\mathrm{F}\mathrm{r}\mathrm{o}\mathrm{b}_{v}$ acts as
$$
q^{-n}\prod_{j=1}^{n}\prod_{i=1}^{N}g(v, a_{i}+b_{j})
$$
under $\det V_{\lambda,0}$.

Considering the choice of $a_{i}$ and $b_{j}$, the product can be rewritten as
\begin{equation}
\begin{split}\label{4.1}
 \prod_{j=1}^{n}\prod_{i=1}^{N}g(v, a_{i}+b_{j})&=(\prod_{j=1}^{n}g(v, b_{j}))^{n}\prod_{j=1}^{n}\prod_{s\neq -b_{k},\forall k}g(v, s+b_{j})\\
 &=(\prod_{j=1}^{n}g(v, b_{j}))^{n}(\prod_{s\neq 0}g(v, s))^{n}/\prod_{i,j\in\{1,\ldots,n\},i\neq j}g(v, b_{j}-b_{i})\\
 &=(\prod_{j=1}^{n}g(v, b_{j}))^{n}(\prod_{s\neq 0}g(v, s))^{n}/\prod_{i,j\in\{1,\ldots,n\},i<j}q\\
 &=(\prod_{j=1}^{n}g(v,\ b_{j}))^{n}(q^{(N-1)/2})^{n}/q^{n(n-1)/2}
\end{split}
\end{equation}
for $v$ whose residue characterisitc is odd, and here $s$ always ranging through the residue class in $\Z/N\Z$. In the last two steps, we use that for any nonzero $a\in \mathbb{Z}/N\mathbb{Z}$,
$$
g(v,a)g(v,-a)=(-1)^{a\frac{1-q}{N}}q=q
$$

We further verify that $\displaystyle \prod_{j=1}^{n}g(v, b_{j})\in \mathbb{Q}_l(\zeta_{N})$ by checking: $\forall\sigma\in G_{\Q_l(\zeta_{N})}$, if $\sigma(\zeta_{p})=\zeta_{p}^{a}$ ($p$ is the residue characteristic of $v$), then
\begin{equation}
\begin{split}\label{4.2eq}
\sigma(\prod_{j=1}^{n}g(v, b_{j}))&=\sigma(\prod_{j=1}^{n}\sum_{x\in \F_{q}^{\times}}-t(x^{\frac{1-q}{N}})^{b_{j}}\psi(x))\\
&=\prod_{j=1}^{n}\sum_{x\in \F_{q}^{\times}}-t(x^{\frac{1-q}{N}})^{b_{j}}\psi(x)^{a}\\
&=\prod_{j=1}^{n}\sum_{x\in \F_{q}^{\times}}-t((a^{-1}x)^{\frac{1-q}{N}})^{b_{j}}\psi(x)\\
&=\prod_{j=1}^{n}t(a^{\frac{q-1}{N}})^{b_{j}}\prod_{j=1}^{n}\sum_{x\in \mathrm{F}_{q}^{\times}}-t(x^{\frac{1-q}{N}})^{b_{j}}\psi(x)\\
&=\prod_{j=1}^{n}g(v, b_{j})
\end{split}
\end{equation}
since $\sum_{j=1}^{n}b_{j}=0$ mod $N$. 

This suffices when $n>2$. When $n=2$, we have to show $\displaystyle \prod_{j=1}^{n}g(v, b_{j})\in \mathbb{Q}_l(\zeta_{N})^+$. For this, it suffices to take a $\sigma\in G_{\mathbb{Q}_l(\zeta_{N})^+}$ such that $\sigma(\zeta_N)=\zeta_N^{-1}$ and $\sigma(\zeta_p)=\zeta_p$, and
\begin{equation}
\begin{split}
\sigma(\prod_{j=1}^{2}g(v, b_{j}))&=\sigma(\prod_{j=1}^{n}\sum_{x\in \F_{q}^{\times}}-t(x^{\frac{1-q}{N}})^{b_{j}}\psi(x))\\
&=\prod_{j=1}^{2}\sum_{x\in \F_{q}^{\times}}-t(x^{\frac{1-q}{N}})^{-b_{j}}\psi(x)\\
&=\prod_{j=1}^{2}\sum_{x\in \F_{q}^{\times}}-t(x^{\frac{1-q}{N}})^{b_{j}}\psi(x)\\
&=\prod_{j=1}^{2}g(v, b_{j})
\end{split}
\end{equation}
because $\{b_1, b_2\}=\{-b_1, -b_2\}$.

Therefore, we deduce that $\det V[\lambda](G_{F})$ lands in $(\mathbb{F}_{l^{2}}^{\times}k(\lambda)^{\times})^{n}$.

\end{proof}

\ 

We now use the comparison theorems to deduce some $p$-adic Hodge theoretic properties of some $V_{\lambda,t}$. Before doing that, let us fix some notation, following \cite{BLGHT}.

Let $\mathcal{H}_{\mathrm{D}\mathrm{R}}$ denote the degree $N-2$ relative de Rham cohomology of $Y$:
$$
\mathcal{H}_{\mathrm{D}\mathrm{R}}=\mathcal{H}_{\mathrm{D}\mathrm{R}}^{N-2}(Y/(T_{0}\times \mathbb{Q}(\zeta_{N})))
$$
It is a locally free sheaf over $T_{0}\times \mathbb{Q}(\zeta_{N})$ with a decreasing filtration $F^{j}\mathcal{H}_{\mathrm{D}\mathrm{R}}$ by local direct summands. For $\sigma\in \mathrm{G}\mathrm{a}\mathrm{l}(\mathbb{Q}(\zeta_{N})/\mathbb{Q})$ , let $\mathcal{H}_{\mathrm{D}\mathrm{R},\sigma}$, $F^{j}\mathcal{H}_{\mathrm{D}\mathrm{R},\sigma}$ be the "twist" of $\mathcal{H}_{\mathrm{D}\mathrm{R}}, F^{j}\mathcal{H}_{\mathrm{D}\mathrm{R}}$ respectively:
$$
\mathcal{H}_{\mathrm{D}\mathrm{R},\sigma}=\mathcal{H}_{\mathrm{D}\mathrm{R}}\otimes_{\sigma^{-1},\mathbb{Q}(\zeta_{N})}\mathbb{Q}(\zeta_{N})\text{ , }F^{j}\mathcal{H}_{\mathrm{D}\mathrm{R},\sigma}=F^{j}\mathcal{H}_{\mathrm{D}\mathrm{R}}\otimes_{\sigma^{-1},\mathbb{Q}(\zeta_{N})}\mathbb{Q}(\zeta_{N})
$$

$H_{0}$ acts on $\mathcal{H}_{\mathrm{D}\mathrm{R}}, F^{j}\mathcal{H}_{\mathrm{D}\mathrm{R}}$ in the usual way. Let $V_{\mathrm{D}\mathrm{R},\sigma}$, $F^{j}V_{\mathrm{D}\mathrm{R},\sigma}$ denote the $\chi$ eigenspace of $\mathcal{H}_{\mathrm{D}\mathrm{R}}, F^{j}\mathcal{H}_{\mathrm{D}\mathrm{R}}$:
$$
V_{\mathrm{D}\mathrm{R},\sigma}=(\mathcal{H}_{\mathrm{D}\mathrm{R},\sigma})^{\chi,H_{0}},\ F^{j}V_{\mathrm{D}\mathrm{R},\sigma}=(F^{j}\mathcal{H}_{\mathrm{D}\mathrm{R},\sigma})^{\chi,H_{0}}
$$
where we view $\mathcal{H}_{\mathrm{D}\mathrm{R},\sigma}$, $F^{j}\mathcal{H}_{\mathrm{D}\mathrm{R},\sigma}$ as $\mathbb{Q}(\zeta_{N})$ vector space by acting on the right. Let $\mathrm{g}\mathrm{r}^{j}V_{\mathrm{D}\mathrm{R},\sigma}=F^{j}V_{\mathrm{D}\mathrm{R},\sigma}/F^{j+1}V_{\mathrm{D}\mathrm{R},\sigma}$ be the associated graded pieces.

Again $H$ acts on $\mathcal{H}_{\mathrm{D}\mathrm{R},\sigma, 0}$, $F^{j}\mathcal{H}_{\mathrm{D}\mathrm{R},\sigma,0}$, and we have:
$$
V_{\mathrm{D}\mathrm{R},\sigma,0}=\bigoplus_{i=1}^{N}V_{\mathrm{D}\mathrm{R},\sigma,i},\ F^{j}V_{\mathrm{D}\mathrm{R},\sigma,0}=\bigoplus_{i=1}^{N}F^{j}V_{\mathrm{D}\mathrm{R},\sigma,i}
$$
here $V_{\mathrm{D}\mathrm{R},\sigma,i}$, $F^{j}V_{\mathrm{D}\mathrm{R},\sigma,i}$ are the subspace of  $V_{\mathrm{D}\mathrm{R},\sigma,0}$, $F^{j}V_{\mathrm{D}\mathrm{R},\sigma,0}$ resp.  where $H$ acts by:
$$
\xi\rightarrow\prod_{j=1}^{N}\xi_{j}^{a_{j}+i}
$$
and we let $\mathrm{g}\mathrm{r}^{j}V_{\mathrm{D}\mathrm{R},\sigma,i}=F^{j}V_{\mathrm{D}\mathrm{R},\sigma,i}/F^{j+1}V_{\mathrm{D}\mathrm{R},\sigma,i}$ be the associated graded pieces. Let $\lambda$ and $v$ be primes of $\mathbb{Z}[\zeta_N]$ both of characteristic $l$. Here $\lambda$ is the prime of the coefficients field as before and $v$ is the place we will restrict to in the $p$-adic Hodge theory setting. If $F$ is a finite extension over $\mathbb{Q}(\zeta_{N})_{v}$ and $t\in T_{0}(F)$, for an embedding $\sigma:F\inj\overline{\mathbb{Q}(\zeta_{N})}_{\lambda}$, by the etale comparison theorem, we have
\[
((H^{N-2}_{et}(Y_{t}\times\overline{F},\ \mathbb{Z}[\zeta_{N}]_{\lambda})\otimes_{\mathbb{Z}[\zeta_{N}]_{\lambda}}\overline{\mathbb{Q}(\zeta_{N})}_{\lambda})\otimes_{\sigma,F}B_{\mathrm{D}\mathrm{R}})^{\gal(\overline{F}/F)}\cong H_{\mathrm{D}\mathrm{R}}^{N-2}(Y_{t}/F)\otimes_{F,\sigma}\overline{\mathbb{Q}(\zeta_{N})}_{\lambda}
\]
as filtered vector space. Taking the $\chi$ eigenspace of $H_{0}$ action on both sides gives (notice the twist):
$$
((V_{\lambda,t}\otimes_{\mathbb{Z}[\zeta_{N}]_{\lambda}}\overline{\mathbb{Q}(\zeta_{N})}_{\lambda})\otimes_{\sigma,F}B_{\mathrm{D}\mathrm{R}})^{\gal(\overline{F}/F)}\cong V_{\mathrm{D}\mathrm{R},\sigma_{|\Q(\zeta_{N})},t}\otimes_{F,\sigma}\overline{\mathbb{Q}(\zeta_{N})}_{\lambda}
$$
as filtered vector space.

Similarly, for $i\in\{1,2,\ldots, N\}$ and $\sigma:\mathbb{Q}(\zeta_{N})_{v}\inj\overline{\mathbb{Q}(\zeta_{N})}_{\lambda}$, view $0\in T_{0}(\mathbb{Q}(\zeta_{N})_{v})$, we have:
\[
((V_{\lambda,i}\otimes_{\mathbb{Z}[\zeta_{N}]_{\lambda}}\overline{\mathbb{Q}(\zeta_{N})}_{\lambda})\otimes_{\sigma,\mathbb{Q}(\zeta_{N})_{v}}B_{\mathrm{D}\mathrm{R}})^{\gal(\overline{\mathbb{Q}(\zeta_{N})_{v}}/\mathbb{Q}(\zeta_{N})_{v})}\cong V_{\mathrm{D}\mathrm{R},\sigma_{|\Q(\zeta_{N})},i}\otimes_{\mathbb{Q}(\zeta_{N})_{v},\sigma}\overline{\mathbb{Q}(\zeta_{N})}_{\lambda}
\]

For $a\in \mathbb{Z}/N\mathbb{Z}$, we will write $\overline{a}$ as the representative element in the range $\{$1, 2, \ldots, $N\}$ of $a$. Let $\tau_{0}:\mathbb{Q}(\zeta_{N})\inj\mathbb{C}$ be the  embedding $: \zeta_N\mapsto e^{2\pi i/n}$. Assume $\sigma^{-1}(\zeta_{N})=\zeta_{N}^{a}$.

\begin{lemma}
\label{3.7}
{\it Under the notation and assumption above, we have}
\begin{enumerate}
    \item $V_{DR,\sigma,i}\neq(0)$ {\it only when} $i\in\{b_{1},\ldots, b_{n}\}$. {\it And for each such $i$}, $V_{DR,\sigma,i}$ {\it is a one-dimensional} $\Q(\zeta_N)$-{\it vector space and}  $\mathrm{g}\mathrm{r}^{j}V_{DR,\sigma,i}\neq 0$ {\it only when}
$$
j=M(a)+\#\{b\in\{b_{1},\ldots, b_{n}\}: \overline{ab}<\overline{ai}\}
$$
{\it here} $M(a)$ {\it is some constant determined by} $a$.

\item $\mathrm{g}\mathrm{r}^{j}V_{DR,\sigma}$ {\it is locally free of rank} 1 {\it over}  $T_{0}\times \mathbb{Q}(\zeta_{N})$ {\it when} $M(a)\leq j\leq M(a)+ n-1$ {\it and is} (0) {\it otherwise}. $V_{DR,\sigma}$ {\it is a locally free sheaf over} $T_{0}\times \mathbb{Q}(\zeta_{N})$ {\it of  rank} $n$.
\end{enumerate}
\end{lemma}

\begin{proof}
Base change to $\mathbb{C}$ gives that
$$
\mathrm{g}\mathrm{r}^{j}V_{\mathrm{D}\mathrm{R},\sigma,i}\otimes_{\mathbb{Q}(\zeta_{N}),\tau_{0}\sigma^{-1}}\mathbb{C}\cong H^{j,N-2-j}(Y_{0}(\mathbb{C}),\ \mathbb{C})_{(a(a_{1}+i),\ldots,a(a_{N}+i))}
$$
where we define $Y(\mathbb{C})$ via the embedding $\tau_{0}$ and right hand side of the isomor-phism is defined as the eigenspace of $H^{j,N-2-j}(Y_{0}(\mathbb{C}), \mathbb{C})$ where $H$ acts by $\xi\rightarrow  \prod_{j=1}^{N}\xi_{j}^{a(a_{j}+i)}$. Proposition I.7.4 and I.7.6 of \cite{DMOS} gives that right hand side is nonzero if and only if

\begin{itemize}
    \item  indices $a(a_{1}+i),\ldots, a(a_{N}+i)$ are all nonzero mod $N$

    \item and
$$
j+1=(a(a_{1}+i)+\ldots+a(a_{N}+i))/N
$$
\end{itemize}
i.e. $i\in\{b_{1},\ldots, b_{n}\}$ and we derive the formula of $j$ for a fixed such $i$ as below. For $1\leq d\leq N$, let
$$
j(d)=(\overline{aa_{1}+d}+\ldots +\overline{aa_{N}+d})/N-1
$$

Note that the nonzero $a_{i}$ only appears once in the sum. Then $j(d+1)=j(d)+1$ if $d\equiv ab_{j}$ mod $N$ for some $b_{j}$ (in this case, none of $\overline{aa_{i}+d}$ is $N$), and $j(d+1)=j(d)$ if otherwise (in this case, exactly one of $\overline{aa_{i}+d}$ is $N$). Use this formula to induct,we see that taking $M(a)=j(1)$ gives the formula in (1).

Since $\mathrm{g}\mathrm{r}^{j}V_{\mathrm{D}\mathrm{R},\sigma}$ is locally free, it suffice to look at the fibre over $0$. Lining up $\overline{ab_{i}}$ in order, we see that the $j$ such that $\mathrm{g}\mathrm{r}^{j}V_{\mathrm{D}\mathrm{R},\sigma,0}\neq 0$ are precisely $M(a),\ldots, M(a)+ n-1$. (2) follows immediately.

\end{proof}

\begin{lemma}
\label{3.8}
{\it Under the notation and assumption above, and let} $t\in F$ {\it as a point in} $T_{0}(F)$ $\zeta_N\in F$, {\it we have}

\begin{enumerate}
    \item  $V_{\lambda,t}$ {\it is a de Rham representation of} $G_{F}$. {\it For} $\sigma:F\inj\overline{\mathbb{Q}(\zeta_{N})}_{\lambda}$, {\it if} $\sigma^{-1}(\zeta_{N})=\zeta_{N}^{a}$, {\it then the Hodge-Tate weight of} $V_{\lambda,t}\otimes_{\mathbb{Z}[\zeta_{N}]_{\lambda}}\overline{\mathbb{Q}(\zeta_{N})_{\lambda}}$ {\it with respect to} $\sigma$ {\it are}
$$
\{M(a),\ M(a)+1,\ldots, M(a)+n-1\}.
$$

    \item {\it If} $t\in \mathcal{O}_{F}$ {\it and} $t^{N}-1\not\in \m_{F}$, {\it then} $V_{\lambda,t}$ {\it is crystalline}.

     \item {\it If} $l\equiv 1$ {\it mod} $N$ {\it and} $t\in \m_{F}$, {\it then} $V_{\lambda,t}$ {\it is ordinary of weight} $(\lambda_{\sigma,i})$ {\it with} $\lambda_{\sigma,i}=M(a_{\sigma}),\ \forall i$, where $a_{\sigma}$ satisfy $\sigma^{-1}(\zeta_{N})=\zeta_{N}^{a_{\sigma}}$ .
     
     \item If $v(t)<0$, {\it then} $V_{\lambda,t}$ {\it is regular and ordinary of weight} $(\lambda_{\sigma,i})$ {\it with} $\lambda_{\sigma,i}=M(a_{\sigma}),\ \forall i$, where $a_{\sigma}$ satisfy $\sigma^{-1}(\zeta_{N})=\zeta_{N}^{a_{\sigma}}$ .

\end{enumerate}
\end{lemma}

\begin{proof}
(1) is clear from the comparison theorem and Lemma \ref{3.7}.

(2) follows from the fact that these $Y_{t}$ have good reduction modulo the maximal ideal of $F$.

(3) we observe that since $t\in \m_{F}$,
$$
(V_{\lambda,0}\otimes_{\sigma, F}B_{\mathrm{c}\mathrm{r}\mathrm{i}\mathrm{s}})^{\gal(\overline{F}/F)}\cong(V_{\lambda,t}\otimes_{\sigma, F}B_{\mathrm{c}\mathrm{r}\mathrm{i}\mathrm{s}})^{\gal(\overline{F}/F)}
$$
as $\phi$-module because they can both be written as the $\chi$-eigenspace of the crystalline cohomology of the reduction of $Y_{0}$. Moreover, the Hodge-Tate weights of $V_{\lambda,0}$ and $V_{\lambda,t}$ are the same by (1). Thus by Lemma \ref{2.3}, $V_{\lambda,t}$ is ordinary of weight $(M(a_\sigma))$ if and only if $V_{\lambda,0}$ is ordinary of weight $(M(a_\sigma))$ .

Recall $V_{\lambda,0}=\oplus_{i=1}^{N}V_{\lambda,i}$ as $G_{\mathbb{Q}(\zeta_{N})_{v}}=G_{\mathbb{Q}_{l}}$ representation. They are both $\mathbb{Q}(\zeta_{N})_{\lambda}=\mathbb{Q}_{l}$ vector space. Since
\[
((V_{\lambda,i}\otimes_{\mathbb{Z}[\zeta_{N}]_{\lambda}}\overline{\mathbb{Q}(\zeta_{N})}_{\lambda})\otimes_{\sigma,\mathbb{Q}(\zeta_{N})_{v}}B_{\mathrm{D}\mathrm{R}})^{\gal(\overline{\mathbb{Q}(\zeta_{N})_{v}}/\mathbb{Q}(\zeta_{N})_{v})}\cong V_{\mathrm{D}\mathrm{R},\sigma_{|\mathbb{Q}(\zeta_{N})},i}\otimes_{\mathbb{Q}(\zeta_{N})_{v},\sigma}\overline{\mathbb{Q}(\zeta_{N})}_{\lambda}
\]
$V_{\lambda,i}$ is $1$-dimensional when $i\in\{b_{1},\ldots, b_{n}\}$ and has Hodge-Tate weight $M(a_\sigma)+ \#\{b\in\{b_{1},\ldots, b_{n}\}: \overline{a_\sigma b}<\overline{a_\sigma i}\}$ in this case. Thus
$$
V_{\lambda,0}\otimes_{\mathbb{Q}_{l}}\overline{\mathbb{Q}}_{l}\cong\bigoplus_{i=0}^{n-1}\overline{\mathbb{Q}}_{l}(-M(a_\sigma)-i)
$$
as $I_{\mathbb{Q}(\zeta_{N})_{v}}=I_{\mathbb{Q}_{l}}$-representation. Therefore (3) follows.

(4) This is the main theorem of \cite{qian}.

\end{proof}

\section{Proof of Main Result Theorem \ref{1.1}}

We fix a non-CM elliptic curve $E/\mathbb{Q}$. For any prime $l'$, let $\overline{r}_{E, l'}$ be the $G_\Q$ representation $H^1_{et}(E, \F_{l'})$. Write $n=l^{a}m, l\nmid m$

We could find (by Lemma \ref{2N}) a positive integer $N$ satisfying the following properties related only to $\overline{r}, F^{\mathrm{a}\mathrm{v}}$ and $n$ as given in Theorem \ref{1.1}.

\begin{itemize}
    \item $N$ is odd, and is not divisible by any prime factors of $ln$, any prime that is ramified in $F^{\mathrm{a}\mathrm{v}}$ and $F^{\mathrm{k}\mathrm{e}\mathrm{r}\overline{r}}$ and any prime where the elliptic curve $E$ has bad reduction.

    \item  $N>100n+100$

    \item $\mathbb{F}_{l^{2}}\mathbb{F}'\subset \mathbb{F}_{l}(\zeta_{N})$ , where $\mathbb{F}'$ is the finite field generated by all the $m$-th roots(hence {\it n}-th roots) of elements in the field $\mathbb{F}_{l^{s}}$ we choose such that the residual representation $\overline{r}: G_F\rightarrow GL_n(\mathbb{F}_{l^{s}})$. And when $n=2$, we further want that $\F_l(\zeta_N)=\F_{l^r}$ for some $r$ even and $\mathbb{F}_{l^{2}}\mathbb{F}'\subset \mathbb{F}_{l}(\zeta_{N})^+$. These all amounts to the condition that the smallest positive integer $r$ such that $N\mid l^r-1$ is divisible by certain integers.

    \item Let $\F_l(\zeta_N)=\F_{l^r}$. When $r$ is even, we have $N\nmid l^{r/2}+1$.

\end{itemize}

Set  $F^{\mathrm{a}\mathrm{v}\mathrm{o}\mathrm{i}\mathrm{d}}$ to be the normal closure over $\mathbb{Q}$ of $F^{\mathrm{a}\mathrm{v}}\overline{F}^{\mathrm{K}\mathrm{e}\mathrm{r}\overline{r}}(\zeta_l)$. Thus by the condition above, $\Q(\zeta_N)$ and $F^{\av}$ are linearly disjoint over $\Q$, since any rational prime $p$ that is ramified in their intersection has to divide $N$ while also ramified to $F^\av$. Such prime does not exist, so their intersection is unramified over $\Q$ and thus must be $\Q$. Hence $F^{\mathrm{a}\mathrm{v}\mathrm{o}\mathrm{i}\mathrm{d}}$ and $F(\zeta_{N})$ are linearly disjoint over $F$. Following the proof of Corollary 7.2.4 of \cite{tap}, we can prove the following statement:

\begin{proposition}
\label{po4.1}
{\it In the above notation, there exists a rational prime $l'$  such that}:
\begin{itemize}
    \item $ l'\equiv 1$ {\it mod} $N$

    \item $l'>2ln+5$ and is unramified in $F$.

    \item $\overline{r}_{E,l'}(G_{\widetilde{F}})=GL_{2}(\mathbb{F}_{l'})$, {\it here} $\widetilde{F}$ {\it is the normal closure of $F$ over} $\mathbb{Q}$.

    \item $\exists\sigma\in G_{F}-G_{F(\zeta_{l'})}$ {\it such that} $\overline{r}_{E,l'}(\sigma)$ {\it is a scalar}

    \item $E$ {\it has good ordinary reduction at} $l'$.
\end{itemize}

{\it And there exists a} fi{\it nite Galois extension} $F_{2}^{\mathrm{a}\mathrm{v}\mathrm{o}\mathrm{i}\mathrm{d}}/\mathbb{Q}$ {\it and a finite totally real Galois  extension} $F^{\mathrm{s}\mathrm{u}\mathrm{f}\mathrm{f}}/\mathbb{Q}$ {\it unramified above the prime divisors of} $N$ {\it such that}:
\begin{itemize}
    \item  $F_2^\av\cap F^\av=\mathbb{Q}$

    \item $F^\suff\cap F^\av F_2^\av=\mathbb{Q}$
    \item $\overline{\mathbb{Q}}^{\mathrm{k}\mathrm{e}\mathrm{r}\overline{r}_{E,l'}}\subset F_{2}^{\mathrm{a}\mathrm{v}\mathrm{o}\mathrm{i}\mathrm{d}}$

    \item $F^{\mathrm{a}\mathrm{v}\mathrm{o}\mathrm{i}\mathrm{d}}$ {\it and} $F_{2}^{\mathrm{a}\mathrm{v}\mathrm{o}\mathrm{i}\mathrm{d}}$ {\it are unramified above} prime divisors of $N$
\end{itemize}

{\it and for any finite totally real extension} $F'/F^{\mathrm{s}\mathrm{u}\mathrm{f}\mathrm{f}}$ {\it such that} $F'\cap F_{2}^{\mathrm{a}\mathrm{v}\mathrm{o}\mathrm{i}\mathrm{d}}=\mathbb{Q}$, $\mathrm{S}\mathrm{y}\mathrm{m}\mathrm{m}^{n-1}r_{E,l'}|_{G_{F'}}$ {\it is automorphic}.

\end{proposition} 

\begin{proof}
We first pick an $l'$ satisfying the listed properties. This can be done because the first condition give a set of primes of positive density. The second, third and fifth condition exclude a set of primes of density $0$(the third by \cite{Ser72} and the fifth by \cite{Ser81} Theorem 20), while the fourth condition follows from the second and the third.(Just pick a $u\in \F_{l'}^\times$ with $u^2\neq 1$ and $\sigma\in G_{\widetilde{F}}$ such that $\overline{r}_{E, l'}(\sigma)=u$.) 

Carry out the proof of Corollary 7.2.4 of \cite{tap} to $F^{\mathrm{a}\mathrm{v}\mathrm{o}\mathrm{i}\mathrm{d}}= F^{\mathrm{a}\mathrm{v}\mathrm{o}\mathrm{i}\mathrm{d}}, E=E, \mathcal{M}=\{n-1\}, \mathcal{L}=$ \{the prime divisors of $N$\} and take the $l$ in the proof to be the rational prime $l'$ we just picked. Note that the properties of $l'$ listed in our proposition implies all the properties of $l$ needed in the first paragraph of proof of Corollary 7.2.4 of \cite{tap}.

Inspecting the proof closely would give that the additional properties (the third and fourth of the lower bullet list) also hold:

The third property follows from the choice of $F_{2}^{\mathrm{a}\mathrm{v}\mathrm{o}\mathrm{i}\mathrm{d}}=F_{1}^{\mathrm{a}\mathrm{v}\mathrm{o}\mathrm{i}\mathrm{d}}L_{3}$ in the 12th line of page 183 and $L_{3}=L_{2}\overline{\mathbb{Q}}^{\mathrm{k}\mathrm{e}\mathrm{r}\overline{r}_{E,l'}}$ in the 14th line of page 182.

For the fourth property, $F^{\mathrm{a}\mathrm{v}\mathrm{o}\mathrm{i}\mathrm{d}}$ unramified over $\mathcal{L}$ follows from the choice of $N$ and $F_{2}^{\mathrm{a}\mathrm{v}\mathrm{o}\mathrm{i}\mathrm{d}}$ unramified above $\mathcal{L}$ follows from

\begin{enumerate}
    \item  $L_{3}$ unramified above $\mathcal{L}$  since $\overline{\mathbb{Q}}^{\mathrm{k}\mathrm{e}\mathrm{r}\overline{r}_{E,l'}}$ is unramified above $\mathcal{L}$ and that each $\overline{\mathbb{Q}}^{\mathrm{k}\mathrm{e}\mathrm{r}\mathrm{I}\mathrm{n}\mathrm{d}_{G_{\mathbb{Q}}}^{G_{L}}\overline{\psi}_{m}}$ is unramified over $\mathcal{L}$ ($\psi_{m}$ is unramified over $\mathcal{L}$, see the properties of $\psi_m$ in the beginning of Page 182 and $L$ is also unramified above $\mathcal{L}$, see the paragraph before the last paragraph in Page 181) gives that their composite $L_{2}$ (page 182) is unramified over $\mathcal{L}$.

    \item $F_{1}^{\mathrm{a}\mathrm{v}\mathrm{o}\mathrm{i}\mathrm{d}}$ is obtained from applying Proposition 7.2.3 of \cite{tap} to $F=F_{0}=\mathbb{Q}$ (the $F$ is that in the proposition, not our $F$), $\{r_{m},\ m\in \mathcal{M}\}$, $\mathcal{L}$ and $F^{\mathrm{a}\mathrm{v}\mathrm{o}\mathrm{i}\mathrm{d}}L_{3}$ (See the second paragraph of Page 183). In terms of the proof of Proposition 7.2.3 of \cite{tap}, $F_{1}^{\mathrm{a}\mathrm{v}\mathrm{o}\mathrm{i}\mathrm{d}}=\overline{\mathbb{Q}}^{\mathrm{k}\mathrm{e}\mathrm{r}\prod_{i}\overline{r}_{i}'}(\zeta_{l'})$ with $\overline{r}_{i}'$ unramified above $\mathcal{L}$ and $l'\notin \mathcal{L}$ (See the last line of Page 180).
\end{enumerate}

\end{proof}

Note that the condition on $N, l, l'$ guarantess the hypothesis of $N, l, l'$ in section \ref{3} are all satisfied.

Now, apply Lemma \ref{2.1} to $F=F, M=F(\zeta_{N}), F_{0}=F^{\mathrm{a}\mathrm{v}\mathrm{o}\mathrm{i}\mathrm{d}}F_{2}^{\mathrm{a}\mathrm{v}\mathrm{o}\mathrm{i}\mathrm{d}}F^{\mathrm{s}\mathrm{u}\mathrm{f}\mathrm{f}}(\zeta_N)$,  to see we may take a finite CM Galois extension $E/F$ with $E=LM$ for some totally real Galois extension $L/\Q$ such that $L$ and $F_0$ are linearly disjoint over $\Q$, and that we may find some characters $\chi_{1}$ : $\gal(\overline{F}/E)\rightarrow(\mathbb{Z}[\zeta_{N}]/\lambda)^{\times}$ and $\chi_{2}$ : $\gal(\overline{F}/E)\rightarrow(\mathbb{Z}[\zeta_{N}]/\lambda')^{\times}$ such that $(\overline{\chi_{1}}\times\overline{\chi_{2}})^{n}\cong(\det V[\lambda\lambda'])\otimes\det(\overline{r}\times \mathrm{S}\mathrm{y}\mathrm{m}\mathrm{m}^{n-1}r_{E,l'})^{\vee}$ as $G_{E}$-module. The condition of Lemma \ref{2.1} is verified below.

On the side of characteristic $l$, fix a prime $\lambda$ of $\Q(\zeta_N)$ ($\Q(\zeta_N)^+$ when $n=2$) and denote the residue field by $k(\lambda)$. Note that the condition of $N$ in the beginning of this section gives that $\det\overline{r}$ actually has image landing in $(k(\lambda)^{\times})^{n}$. We also assumed that $\F_{l^2}\subset k(\lambda)$, $n\mid \#(k(\lambda)^\times)$. Hence, applying Lemma \ref{npower} to the field $F(\zeta_N)$, we see that on the characteristic $l$ side, we have $\det V[\lambda]\otimes (\det \overline{r})^\vee$ (as a $G_{F(\zeta_N)}$ representation) has image in $(k(\lambda)^{\times})^{n}$.

On the other side of characteristic $l'$, fixing a prime $\lambda'$ of $\mathbb{Z}[\zeta_{N}]$ ($\mathbb{Z}[\zeta_{N}]^+$ when $n=2$) over $l'$, we have that $\det \mathrm{S}\mathrm{y}\mathrm{m}\mathrm{m}^{n-1}\overline{r}_{E,l'}=(\chi_{\mathrm{c}\mathrm{y}\mathrm{c}})^{n(n-1)/2}$, which have image in $\mathbb{F}_{l^2}^\times$. Applying Lemma \ref{npower} to the prime $l'$ to see that  $(\det\text{Symm}^{n-1}\overline{r}_{E,l'})^{\vee}\otimes (\det V[\lambda'])$ (as a $G_{F(\zeta_N)}$ representation) has image in $(k(\lambda')^{\times})^{n}$.

Let $W$ be the $\mathbb{Z}[\zeta_{N}]/\lambda\lambda'$-module with a $G_E$ action given by the representation $(\overline{\chi_{1}}\otimes\overline{r})\times(\overline{\chi_{2}}\otimes \mathrm{S}\mathrm{y}\mathrm{m}\mathrm{m}^{n-1}\overline{r}_{E,l'})$. The isomorphism $(\overline{\chi_{1}}\times\overline{\chi_{2}})^{n}\cong(\det V[\lambda\lambda'])\otimes\det(\overline{r}\times \mathrm{S}\mathrm{y}\mathrm{m}\mathrm{m}^{n-1}r_{E,l'})^{\vee}$ induces an isomorphism
$$
\phi:\bigwedge^n W_{(T_{0})_{E}}\rightarrow\bigwedge^n V[\lambda\lambda']
$$
. In this way, the moduli functor $T_{W}$ is well-defined by $\phi$ over $E.$

We see that the conditions of Proposition \ref{3.6} are satisfied for $N$ and $l, l'$. Thus $T_{W}$ is geometrically connected.

Note that $F^{\mathrm{a}\mathrm{v}\mathrm{o}\mathrm{i}\mathrm{d}}F_{2}^{\mathrm{a}\mathrm{v}\mathrm{o}\mathrm{i}\mathrm{d}}F^{\mathrm{s}\mathrm{u}\mathrm{f}\mathrm{f}}$ and $M=F(\zeta_N)$ are linearly disjoint over $F$ because $F^{\mathrm{a}\mathrm{v}\mathrm{o}\mathrm{i}\mathrm{d}}F_{2}^{\mathrm{a}\mathrm{v}\mathrm{o}\mathrm{i}\mathrm{d}}F^{\mathrm{s}\mathrm{u}\mathrm{f}\mathrm{f}}$ and $\mathbb{Q}(\zeta_{N})$ linearly disjoint over $\mathbb{Q}$, which in turn comes from $F^{\mathrm{a}\mathrm{v}\mathrm{o}\mathrm{i}\mathrm{d}}, F_{2}^{\mathrm{a}\mathrm{v}\mathrm{o}\mathrm{i}\mathrm{d}}, F^{\mathrm{s}\mathrm{u}\mathrm{f}\mathrm{f}}$ all unramified over the prime divisors of $N$.

Since $L$ is linearly disjoint with $F_0$ over $\Q$ and $M\subset F_0$, we have that $E=LM$ and $F_0$ are linearly disjoint over $M$. Now $E$ and $F^{\mathrm{a}\mathrm{v}\mathrm{o}\mathrm{i}\mathrm{d}}F_{2}^{\mathrm{a}\mathrm{v}\mathrm{o}\mathrm{i}\mathrm{d}}F^{\mathrm{s}\mathrm{u}\mathrm{f}\mathrm{f}}$ are linearly disjoint over $F$, because $E\cap F^{\mathrm{a}\mathrm{v}\mathrm{o}\mathrm{i}\mathrm{d}}F_{2}^{\mathrm{a}\mathrm{v}\mathrm{o}\mathrm{i}\mathrm{d}}F^{\mathrm{s}\mathrm{u}\mathrm{f}\mathrm{f}}=LM\cap F_0\cap F^{\mathrm{a}\mathrm{v}\mathrm{o}\mathrm{i}\mathrm{d}}F_{2}^{\mathrm{a}\mathrm{v}\mathrm{o}\mathrm{i}\mathrm{d}}F^{\mathrm{s}\mathrm{u}\mathrm{f}\mathrm{f}}=M\cap F^{\mathrm{a}\mathrm{v}\mathrm{o}\mathrm{i}\mathrm{d}}F_{2}^{\mathrm{a}\mathrm{v}\mathrm{o}\mathrm{i}\mathrm{d}}F^{\mathrm{s}\mathrm{u}\mathrm{f}\mathrm{f}}=F$

We will need a  theorem of Moret-Bailly from \cite{HSBT}.

\begin{proposition}
\label{4.2}
{\it Let} $F$ {\it be a number} fi{\it eld and let} $S=S_{1}\amalg S_{2}\amalg S_{3}$ {\it be a} fi{\it nite set  of places of}  $F$, {\it so that every element of} $S_{2}$ {\it is non-archimedean. Suppose that} $T/F$ {\it is a smooth, geometrically connected variety. Suppose also that}
\begin{itemize}
    \item {\it for} $v\in S_{1}$ , $\Omega_{v}\subset T(F_{v})$ {\it is a non-empty open subset} ({\it for the} $v$-{\it topology})
    \item {\it for} $v\in S_{2}$ , $\Omega_{v}\subset T(F_{v}^{\text{nr}})$ {\it is a non-empty open} $\mathrm{G}\mathrm{a}\mathrm{l}(F_{v}^{nr}/F_{v})$-{\it invariant  subset}.

    \item {\it for} $v\in S_{3}, \Omega_{v}\subset T(\overline{F}_{v})$ {\it is a non-empty open} $\gal(\overline{F}_{v}/F_{v})$-{\it invariant subset}.
\end{itemize}

Suppose finally that $H/F$ is a finite Galois extension. Then there is a finite Galois extension $F'/F$  and a point $P\in T(F')$  such that:

\begin{itemize}
    
    \item $F'/F$ {\it is linearly disjoint from} $H/F$

    \item {\it every place} $v$ {\it of} $S_{1}$ {\it splits completely in} $F'$ {\it and if} $w$ {\it is a prime of} $F'$ {\it above} $v$, {\it then} $P\in\Omega_{v}\subset T(F_{w}')$

    \item {\it every plae} $v$ {\it of} $S_{2}$ {\it is unramified in} $F'$ {\it and if} $w$ {\it is a prime of} $F'$ {\it above $v$,then} $P\in\Omega_{v}\cap T(F_{w}')$

    \item {\it if} $w$ {\it is a prime of} $F'$ {\it above some $v$} $\in S_{3}$, {\it then} $P\in\Omega_{v}\cap T(F_{w}')$.
\end{itemize}
\end{proposition} 

\

Let $F^{+}\subset F, E^{+}\subset E$ be the maximal totally real subfield respectively. We apply Proposition \ref{4.2} to  the smooth geometrically connected variety $T={\rm Res}_{\Q}^{EF^{\mathrm{s}\mathrm{u}\mathrm{f}\mathrm{f}}}T_{W}$ defined over $\Q$. We take $H=F_0L=F_0E$.

We take $S_{1}=\{\infty\}$,  $S_{2}=\emptyset$ and $S_{3}=\{l, l'\}$. For $v\in S_1$, we take $\Omega_{v}={\rm Res}_{\Q}^{EF^{\mathrm{s}\mathrm{u}\mathrm{f}\mathrm{f}}}T_{W}(\R)$, i.e. the whole set which is clearly open and non-empty since each copy of $T_W(\C)$ are non-empty. For $v\in S_{3}$, there exists an algebraic morphism $p: T\rightarrow \res^{EF^\suff}_{\Q}T_0$ and we define 
\[
\Omega_{l,0}=\{t=(t_\tau)\in \res^{EF^\suff}_{\Q}T_0(\overline{\Q}_l)=\prod_{\tau: EF^\suff\inj\overline{\Q}_l}T_{0, \tau}(\overline{\Q}_l) \ |\  v_l(t_\tau)<0, \forall \tau\}
\]
\[
\Omega_{l', 0}=\{t=(t_\tau)\in \res^{EF^\suff}_{\Q}T_0(\overline{\Q}_{l'})=\prod_{\tau: EF^\suff\inj\overline{\Q}_{l'}}T_{0, \tau}(\overline{\Q}_{l'})\  |\  v_{l'}(t_\tau)>0, \forall \tau\}
\]
and we define $\Omega_l=p^{-1}(\Omega_{l, 0})$, $\Omega_{l'}=p^{-1}(\Omega_{l', 0})$. Both sets are clearly open, non-empty and Galois invariant.

Hence, we get a finite totally real Galois extension $L'/\Q$ with $L'$ linearly disjoint with $F_0L$ over $\Q$ and a point $ t\in T_{0}(L'EF^\suff)$ (because $L'$ and $EF^\suff$ are linearly disjoint over $\Q$) such that if we denote $L'EF^\suff$ by $F'$ and $L'E^+F^\suff$ by $(F')^+$ then
\begin{itemize}
    \item  $F'\supset EF^{\mathrm{s}\mathrm{u}\mathrm{f}\mathrm{f}}$ is a CM Galois extension over $F$

\item $\overline{\chi_{1}}^{-1}V[\lambda]_{t}\cong\overline{r}|_{G_{F'}}$

\item $\overline{\chi_{2}}^{-1}V[\lambda']_{t}\cong \overline{\mathrm{S}\mathrm{y}\mathrm{m}\mathrm{m}^{n-1}r_{E,l'}}|_{G_{F'}}$

\item  $v(t)<0$ for all primes $v|l$ of $F'$

\item $v(t)>0$ for all primes $v|l'$ of $F'$
\end{itemize}

Now $ F'\cap F_2^\av=L'EF^\suff\cap F_0L\cap F_2^\av=EF^\suff\cap F_2^\av$ because $L'$ and $F_0L$ are linearly disjoint over $\Q$. Moreover, $EF^\suff\cap F_2^\av=LMF^\suff\cap F_0\cap F_2^\av=MF^\suff \cap F_2^\av$ because $L$ and $F_0$ are linearly disjoint over $\Q$. Furthermore, $MF^\suff \cap F_2^\av=F^\suff F(\zeta_N)\cap F^\suff F^\av F_2^\av\cap F_2^\av=F^\suff F\cap F_2^\av$ because $F^{\mathrm{a}\mathrm{v}\mathrm{o}\mathrm{i}\mathrm{d}}F_{2}^{\mathrm{a}\mathrm{v}\mathrm{o}\mathrm{i}\mathrm{d}}F^{\mathrm{s}\mathrm{u}\mathrm{f}\mathrm{f}}$ and $\mathbb{Q}(\zeta_{N})$ linearly disjoint over $\mathbb{Q}$. Finally, $F^\suff F\cap F_2^\av=F^\suff F\cap F^\av F_2^\av\cap F_2^\av=F\cap F_2^\av\subset F^\av\cap F_2^\av=\Q$ because $F^\suff$ and $F^\av F_2^\av$ are linearly disjoint over $\Q$. Therefore,  we conclude that $F'\supset F^\suff$ and is linearly disjoint with $F_2^\av$ over $\Q$, so we see by Proposition \ref{po4.1} that $\mathrm{S}\mathrm{y}\mathrm{m}\mathrm{m}^{n-1}r_{E,l'}|_{G_{(F')^+}}$ and hence $\mathrm{S}\mathrm{y}\mathrm{m}\mathrm{m}^{n-1}r_{E,l'}|_{G_{F'}}$  is automorphic. Since $F'$ and $\overline{\Q}^{\mathrm{K}\mathrm{e}\mathrm{r}\overline{r}_{E,l'}}(\subset F_{2}^{\mathrm{a}\mathrm{v}\mathrm{o}\mathrm{i}\mathrm{d}})$ are linearly disjoint over $\Q$, we again have
\begin{itemize}
    \item $\overline{r}_{E,l'}(G_{F'})\supset SL_{2}(\mathbb{F}_{l'})$

\item $\exists\sigma\in G_{F'}-G_{F'(\zeta_{l'})}$ such that $\overline{r}_{E,l'}(\sigma)$ is a scalar
\end{itemize}

Note that by similar reasoning as the previous paragraph, we have that $ F'\cap F^\av F_2^\av=L'EF^\suff\cap F_0L\cap F^\av F_2^\av=EF^\suff\cap F^\av F_2^\av=LMF^\suff\cap F_0\cap F^\av F_2^\av=MF^\suff \cap F^\av F_2^\av=F^\suff F(\zeta_N)\cap F^\suff F^\av F_2^\av\cap F^\av F_2^\av=F^\suff F\cap F^\av F_2^\av=F$  and thus $F'$ is linearly disjoint with $F^{\mathrm{a}\mathrm{v}}$ over $F$ as we wanted in the main theorem.

Let $\chi_2: G_E\rightarrow \Q(\zeta_N)^\times$ be the Teichmuller lift of $\overline{\chi}_2$. We would like to apply Theorem 6.1.2 of \cite{tap} to $p=l'$, $\rho=V_{\lambda', t}\otimes \chi_2^{-1}$ and $r_{\iota}(\pi)=\sym^{n-1}r_{E,l'}|_{G_{F'}}$. Clearly $\overline{\rho}\cong \overline{r_{\iota}(\pi)}$. For the residual representation $\sym^{n-1}\overline{r}_{E,l'}|_{G_{F'}}$, the two properties of $\overline{r}_{E,l'}$ listed in \ref{po4.1} gives that it is absolutely irreducible and condition (4) of Theorem 6.1.2 of \cite{tap} is satisfied. Now apply Lemma 7.1.5 (2) of \cite{tap} to $F=F$, $F_1=F'$, $l=l'$, $\overline{r}=\overline{r}_{E,l'}$ and the fact that $F'$ and $\overline{\Q}^{\ker \overline{r}_{E,l'}}$ are linearly disjoint over $\Q$, to see $(\sym^{n-1}\overline{r}_{E, l'})(G_{F'(\zeta_{l'})})=(\sym^{n-1}\overline{r}_{E, l'})(G_{F(\zeta_{l'})})$ is enormous. Apply Lemma 7.1.5 (4) of \cite{tap} to the same situation ($H\subset \widetilde{F}F\overline{\Q}^{\ker \overline{r}_{E, l'}}\subset \widetilde{F}F_2^\av\subset F^\av F_2^\av$, thus $H'\subset F^\av F_2^\av$ since $F^\av$ and $F_2^\av$ are both Galois over $\Q$, and then $F'$ linearly disjoint with $F^\av F_2^\av$ over $F$ is satisfied) to see that $\overline{\rho}$ is decomposed generic. Now $V_{\lambda', t}$ is an regular ordinary representation for any place $v_{l'}\mid l'$ of $F'$ by Lemma \ref{3.8} (3) ($v_{l'}(t)>0$) and thus so is $\rho$. $\sym^{n-1}r_{E, l'}|_{G_{F'}}$ is ordinarily automorphic by our choice of $l'$. Theorem 6.1.2 of \cite{tap} thus gives that $V_{\lambda', t}$ is automorphic as a $G_{F'}$ representation. And so $V_{\lambda, t}$ is automorphic as a $G_{F'}$ representation. Now by Lemma \ref{3.8} (4) ($v_l(t)<0$ for any $v_l\mid l$), we see that $V_{\lambda, t}\otimes \chi_1^{-1}$ is regular ordinary as $G_{F'}$ representation and hence $\overline{r}\mid_{G_{F'}}$ is ordinarily automorphic. This finishes the proof of Theorem \ref{1.1}.

Since there are different versions of \cite{tap} with different Lemma 7.1.5, we state it as the following. The proof is taken from an old version of \cite{tap}:

\begin{lemma}
Suppose $F/\Q$ is a finite extension with normal closure $\widetilde{F}/\Q$ and $n\in\Z_{>0}$. Suppose also that $l>2n+5$ is a rational prime and that $\overline{r}:G_F\rightarrow GL_2(\overline{\F}_l)$ is a continuous representation such that $\overline{r}(G_{\widetilde{F}})\supset SL_2(\F_l)$. Finally assume $F_1/F$ is a finite extension that is linearly disjoint from $\overline{F}^{\ker \overline{r}}$ over $F$. Then:

(2) $(\sym^{n-1}\overline{r})(G_{F(\zeta_l)})$ is enormous.

(4) If $F_1/F$ is Galois and linearly disjoint over $F$ from the normal closure $H'$ of $H=\widetilde{F}\overline{F}^{\ker \ad\overline{r}}$ over $\Q$, then $\sym^{n-1}\overline{r}|_{G_{F_1}}$ is decomposed generic.
\end{lemma}

\begin{proof}
(2) This is Lemma 7.1.5(2) of \cite{tap}.

(4) It suffices to show $\sym^{n-1}\overline{r}_{G_{F_2}}$ is decomposed generic for some finite extension $F_2/F_1$.

We may assume without loss of generality that $F$ is Galois over $\Q$. Now by \cite{DDT} Theorem 2.47(b), $\ad\overline{r}(G_F)=PGL_2(k)$ or $PSL_2(k)$ for some finite extension $k/\F_l$. We first take an at most qurdratic extension $E/F$ such that $\ad\overline{r}(G_E)=PSL_2(k)$. Then by Goursat Lemma, $\gal(\widetilde{E}/E)=(\Z/2\Z)^r$ for some $r\geq 0$. Hence $\overline{F}^{\ker \ad\overline{r}}$ and $\widetilde{E}$ are linearly disjoint over $E$ by an analysis of the simple factors of each Galois group. Thus $\ad\overline{r}(G_{\widetilde{E}})=PSL_2(k)\supset PSL_2(\F_l)$. Now since $E\subset\overline{F}^{\ker\ad\overline{r}}$, $\widetilde{E}\overline{\widetilde{E}}^{\ker\ad\overline{r}}\subset\widetilde{\overline{F}^{\ker\ad\overline{r}}}=H'$, and so its normal closure over $\Q$ is still $H'$. Now $F_1$ and $H$ linearly disjoint over $F$ implies that $E_1:=\widetilde{E}F_1$ is linearly disjoint from $H'$ over $\widetilde{E}$ as well. We therefore can assume (replacing $F$ by $\widetilde{E}$ and $F_1$ by $E_1$) without loss of generality that $\ad\overline{r}(G_F)=PSL_2(k)$ and $F$ is Galois over $\Q$.

Now we choose a sequence of subfields $F=F_0'\subset F_1'\subset \cdots\subset F_s'=F_1$ such that $F_i'$ is Galois over $F_{i-1}'$ and $\gal(F_i'/F_{i-1}')$ is simple. Set $\widetilde{F}_i'$ to be the normal closure of $F_i'$ over $\Q$. Hence $\gal(\widetilde{F}_i'/\widetilde{F}_{i-1}')$ is trivial if and only if $F_i'\subset\widetilde{F}_{i-1}'$ and is of form $\Delta_i^m$ for some $m>0$ where $\Delta_i=\gal(F_i'/F_{i-1}')$ otherwise (Goursat Lemma). Now if $H\cap\widetilde{F}_s'=F$, then we may apply Lemma 7.1.5 (3) of \cite{tap} to conclude.

Otherwise, there should exist a minimal $i$ such that $H\cap \widetilde{F}_i'\neq F$. Since $\gal(H/F)=PSL_2(k)$ is simple, we see that $H\subset \widetilde{F}_i'$. Now minimality gives $\widetilde{F}_{i-1}'\cap H=F$, so that $\Delta_i^m=\gal(\widetilde{F}_i'/\widetilde{F}_{i-1}')\twoheadrightarrow \gal(H/F)$, thus $\Delta_i=PSL_2(k)$, $m>0$. We claim there exists $\sigma\in G_\Q$ such that $\sigma H\subset F_i'\widetilde{F}_{i-1}'$. We may write $\widetilde{F}_i'$ as the composite of $\sigma_jF_i'\widetilde{F}_{i-1}'$, where $\sigma_1,\ldots, \sigma_m$ are elements of $G_\Q$, the fields $\sigma_jF_i'\widetilde{F}_{i-1}'$ is Galois over $\widetilde{F}_{i-1}'$ with Galois group $PSL_2(k)$, and for any two disjoint subsets $I, J$ of $\{1,\ldots, m\}$, the composite of $\sigma_jF_i'\widetilde{F}_{i-1}'$ for $j\in I$ and the composite of $\sigma_jF_i'\widetilde{F}_{i-1}'$ for $j\in J$ are linearly disjoint over $\widetilde{F}_{i-1}'$. Now because $H\subset \widetilde{F}_i'$ and $\widetilde{F}_{i-1}'\cap H=F$, we see that $H\widetilde{F}_{i-1}'$ is a Galois subextension of $\widetilde{F}_i'/\widetilde{F}_{i-1}'$ with Galois group $PSL_2(k)$. We may pick the smallest $j$ such that $E_j:=(\sigma_1F_i')\cdots(\sigma_jF_i')\widetilde{F}_{i-1}'\supset H\widetilde{F}_{i-1}'$. The minimality gives that $E_{j-1}$ as a Galois extension of $\widetilde{F}_{i-1}'$ does not contain $H\widetilde{F}_{i-1}'$, whose Galois group over $\widetilde{F}_{i-1}'$ is simple. It follows that $E_{j-1}$ is linearly disjoint with $H\widetilde{F}_{i-1}'$ over $\widetilde{F}_{i-1}'$. Therefore, restriction map takes $\gal(E_j/E_{j-1})$ onto $\gal(H\widetilde{F}_{i-1}'/\widetilde{F}_{i-1}')$. Observe that $E_{j-1}$ and $\sigma_jF_i'\widetilde{F}_{i-1}'$ are linearly disjoint over $\widetilde{F}_{i-1}'$, hence $\gal(E_j/\widetilde{F}_{i-1}')=\gal(E_j/E_{j-1})\times \gal(E_j/\sigma_jF_i'\widetilde{F}_{i-1}')$. The latter group $\gal(E_j/\sigma_jF_i'\widetilde{F}_{i-1}')$ commutes with $\gal(E_j/E_{j-1})$ inside $\gal(E_j/\widetilde{F}_{i-1}')$. Hence under restriction map, by the surjective result proved above, $\gal(E_j/\sigma_jF_i'\widetilde{F}_{i-1}')$ maps into the center of $\gal(H\widetilde{F}_{i-1}'/\widetilde{F}_{i-1}')$, which is trivial. This gives us that $\sigma_jF_i'\widetilde{F}_{i-1}'\supset H\widetilde{F}_{i-1}'$. Thus taking $\sigma_j$ is sufficient for our claim.

Consider the image of $\gal(\widetilde{F}_{i-1}'F_i'/F_i')$ in $\gal(\sigma H/F)$ under the natural restriction map. The fact $m>0$ gives that $F_i'$ and $\widetilde{F}_{i-1}'$ are linearly disjoint over $F_{i-1}'$, and so $\gal(\widetilde{F}_{i-1}'F_i'/F_i')$ and $\gal(\widetilde{F}_{i-1}'F_i'/\widetilde{F}_{i-1}')$ are commuting subgroups of $\gal(\widetilde{F}_{i-1}'F_i'/F_{i-1}')$. Under the restriction map to $\gal(\sigma H/F)$, since $\sigma H\cap\widetilde{F}_{i-1}'=F$, $\gal(\widetilde{F}_{i-1}'F_i'/\widetilde{F}_{i-1}')$ surjects onto $\gal(\sigma H/F)$, so the image of $\gal(\widetilde{F}_{i-1}'F_i'/F_i')$ lies in the center of $\gal(\sigma H/F)\cong PSL_2(k)$,  and hence is trivial. Therefore, $\sigma H\subset F_i'$, which contradicts the condition that $H'$ and $F_1$ are linearly disjoint over $F$.
\end{proof}

For the proof of Theorem \ref{1.4}, we know from above that $\overline{\chi_{1}}^{-1}V[\lambda]_{t}\cong\overline{r}|_{G_{F'}}$ and $V_{\lambda, t}\otimes \chi_1^{-1}$ is regular ordinary as a $G_{F'}$ representation. Thus, in order to apply Theorem 6.1.2 of \cite{tap}, it suffices to verify that the conditions (3) and (4) of that theorem holds for  $\overline{r}|_{G_{F'}}$. Since $F'$ is linearly disjoint with $\overline{F}^{\ker \overline{r}}\subset F^\av$ over $F$, all conditions except decomposed genericity follows from the corresponding conditions of $\overline{r}$. Now $F'=L'EF^\suff=F L'LF^\suff(\zeta_N)$ is Galois over $F$. And the Galois closure of $\overline{F}^{\ker \overline{r}}(\zeta_l)$ is linearly disjoint with $F'$ over $\Q$ because this Galois closure is contained in $F^\av$ and that $ L'LF^\suff(\zeta_N)\cap F^\av=L'LF^\suff(\zeta_N)\cap LF_0\cap F^\av=LF^\suff(\zeta_N)\cap F^\av=LF^\suff(\zeta_N)\cap F_0\cap F^\av=F^\suff(\zeta_N)\cap F^\av=F^\suff(\zeta_N)\cap F^\suff F^\av F_2^\av\cap F^\av=F^\suff \cap F^\av=\Q$. Now we may apply Lemma 7.1.6 of \cite{tap} to see the decomposed genericity.   Hence we also finish the proof of Theorem \ref{1.4}.

\bibliographystyle{amsalpha}
\bibliography{mybibliography}

{\it Date}: March 11, 2019.

\end{document}